\RequirePackage{fix-cm}
\documentclass[twocolumn]{svjour3}          % twocolumn
\smartqed  % flush right qed marks, e.g. at end of proof
\usepackage{graphicx}

\usepackage[noadjust,sort]{cite}
\usepackage{hhline}
\usepackage{graphicx}
\usepackage{amsmath,amssymb}
\usepackage[caption=false,font=footnotesize]{subfig}
\usepackage{url}

\newcommand{\changesColor}{black}

\usepackage{color}
\usepackage{mathtools}
\mathtoolsset{showonlyrefs=true}
\mathtoolsset{centercolon}

\DeclareMathOperator*{\minimize}{minimize}

\DeclareMathOperator*{\aand}{\,and\,}
\DeclareMathOperator*{\subjectto}{subject\ to}

\newcommand{\figheight}{.7\linewidth}
\newcommand{\tablewidth}{1\linewidth}

\begin{document}

\title{Resource Optimization of Product Development Projects with Time-Varying Dependency Structure
\thanks{This work is funded in part by JSPS KAKENHI Grant Number
18K13777 and the open collaborative research program at National Institute of Informatics (NII) Japan (FY2018).}
}


\author{Masaki Ogura \and Junichi Harada \and Masako Kishida \and Ali Yassine
}


\institute{M.~Ogura \and Junichi Harada\at Division of Information Science, Nara Institute of Science and Technology, 8916-5 Takayama, Ikoma, Nara 630-0192, Japan\\
\email{oguram@is.naist.jp, harada.junichi.hh3@is.naist.jp}%
\and
M.~Kishida\at
Principles of Informatics Research Division, National
Institute of Informatics, Tokyo 101-8430, Japan\\
\email{kishida@nii.ac.jp}
\and 
A.~Yassine \at 
Department of Industrial
Engineering and Management,
American University of Beirut,
Beirut 1107-2020, Lebanon\\
\email{ali.yassine@aub.edu.lb}
}

\date{Received: date / Accepted: date}

\maketitle

\begin{abstract}
Project managers are continuously under pressure to shorten product development durations. One practical approach for reducing the project duration is lessening dependencies between different development components and teams. However, most of the resource allocation strategies for lessening dependencies place the implicit and simplistic assumption that the dependency structure between components is static (i.e., does not change over time). This assumption, however, does not necessarily hold true in all product development projects. In this paper, we present an analytical framework for optimally allocating resources to shorten the lead-time of product development projects having a time-varying dependency structure. We build our theoretical framework on a linear system model of product development processes, in which system integration and local development teams exchange information asynchronously and aperiodically. By utilizing a convexity result from the matrix theory, we show that the optimal resource allocation can be efficiently found by solving a convex optimization problem. We provide illustrative examples to demonstrate the proposed framework. \textcolor{\changesColor}{We also present boundary analyses based on major graph models to provide managerial guidelines for improving empirical PD processes.} 
\end{abstract}

\section{Introduction}

Projects are indispensable and central in most of the industries for performing several types of work~\cite{PMI2013}. For this reason, project management has been one of the major research themes in the field of engineering design during the last half-century. Since modern product structures are becoming increasingly complex due to global competition, technological advancements, and changing customer needs, the management of product development (PD) projects is of fundamental importance in various organizations. Despite this fact, there is still a lack of effective managerial methodologies for achieving PD project goals such as meeting prespecified deadlines, achieving desired performances, and keeping cost requirements~\cite{Cicmil2006,Muller2012}. It is even reported that some of the PD projects achieving their goals are often considered to be not successful by stakeholders~\cite{Cooke-Davies2002,Collyer2009}. Therefore, the development of effective methodologies for the modeling and planning of PD projects is highly anticipated by both managers and stakeholders.

Modularization is a widely adopted approach for effective understanding, management, and characterization of complex PD projects. Modularization consists of the following two steps: establishment of a modular architecture and development of design rules~\cite{Baldwin2000}. In the first step, we divide a product into smaller building blocks called modules. This division is performed in such a way that dependencies among the components within individual modules are maximized while dependencies between modules are minimized. Therefore, modularization allows project managers to facilitate a PD project by independently improving individual modules within the product~\cite{Yu2007,Borjesson2014}.

In practice, various technical, physical, and business constraints prohibit us from ignoring dependencies between modules and simply designing each module separately~\cite{Holtta-Otto2007}. Therefore, in the second step of modularization, we seek for the possibility of minimizing or eliminating dependencies outside module boundaries~\cite{Baldwin2000}. Such interdependencies can be reduced by investing in the design rules defining the connections or relationships between modules~\cite{Loch1998,Lee1997,Ahmadi1999}. The reduction is achieved through either the creation of a higher level design rule in the design structure matrix (DSM) or internalizing the rule within the design of each module~\cite{Baldwin2000}. This process helps us to develop a design rule as a mutual upfront agreement between modules~\cite{Martin2002,Frenken2006}.

On the other hand, the current practice for investing in design rules is often based on project managers' intuition or heuristic rules, which may not necessarily lead to the best outcome \cite{Adler1995,Loch1999}. Therefore, we find in the literature several decision support tools aimed at assisting the managers to cost-efficiently invest in design rules (see, e.g., \cite{Krishnan2001,Browning2007,Serrador2014} and references therein). \textcolor{\changesColor}{A distinctive approach to address this problem can be found in the sequence of works~\cite{Braha2004a,Braha2004,Braha2007}, in which the authors identify the fact that performance, robustness, resilience, and fragility of complex design networks are heavily dependent on their underlying connectivity structures. The authors found, by analyzing empirical PD processes, that large-scale design networks are characterized by heavy-tail degree distributions with highly connected modules, while most modules have small degrees. This finding suggests~\cite{Braha2007} that these degree connectivity patterns of complex design networks can be exploited and incorporated in a resource allocation for suppressing the amplification and propagation of design changes and errors through the engineering network.}

However, most of the aforementioned decision support tools in the literature rely on an implicit assumption that the PD project architecture is static and, therefore, does not change over time. This assumption is not consistent with reality; for example, information hiding between development teams give rise to a time-varying dependency structure, which causes persistent recurrence of problems such that progress oscillates between being on schedule and falling behind~\cite{Yassine2003,Lee1997}. Other sources of the time-variability include asynchronicity and random timing of task updates and information exchange~\cite{Mihm2003,Huberman1993} and fluctuations in the amount of rework that a unit of rework in a module causes to other modules~\cite{Huberman2005a}. In fact, \textcolor{\changesColor}{as firstly discovered in the seminal papers~\cite{Braha2006,Braha2009}}, time-variability of dependency structures or, in general, connectivity structures, can be found not only in PD projects but also in various other contexts such as human contact networks, online social networks, biological, and ecological networks \textcolor{\changesColor}{(see, also, the monographs \cite{Holme2015b,Masuda2016b} for recent surveys on this subject)}.

In this paper, we present an optimization framework for making a cost-efficient investment in design rules when the underlying dependency structure between modules is changing over time. We illustrate our theoretical framework by focusing on the PD project model with asynchronous interaction between system integration and local development teams~\cite{Yassine2003}. By using the stability theory of switched linear systems~\cite{Lin2009} adopted from the systems and control theory, we show that the proposed model is feasible, i.e., the amount of unfinished work converges to zero, if and only if the magnitude of the eigenvalues of a generalized work transformation matrix (WTM) is strictly less than one.

We then propose an analytical framework for optimally weakening the dependencies between different product components for accelerating PD projects. Since organizations seek for minimizing project lead-times under a predefined budget, time-cost trade-offs naturally arise in this context. In this paper, we consider the following two types of time-cost trade-off problems~\cite{Hartmann2010}. The first one is the budget-constrained dependency optimization, where we distribute a fixed amount of resource across the project for minimizing its lead-time. The other problem is the performance-constrained dependency optimization, where we distribute resource for achieving a specified performance objective while minimizing the cost of the resource. We show that both problems can be efficiently solved via \emph{convex optimization} techniques~\cite{Boyd2004}, which is in contrast with other decision support tools relying on, e.g., non-convex optimizations~\cite{Yassine2016}, genetic algorithms~\cite{Peteghem2010,Alcaraz2003}, iterative procedures~\cite{Boctor1996}, and backtracking procedures~\cite{Patterson1990}.

The proposed framework is tailored to a dynamical PD project in which a local team passes their outcome to the system team for integration~\cite{Joglekar2001}. We specifically consider the situation where, although the information update from the local team occurs regularly, the one from the system team occurs only intermittently~\cite{Yassine2003}. Such asymmetry in the communication frequency can arise due to high communication cost or unbalanced integration and development times. A major source of high communication cost is poor communications between teams, which could even result in failed projects. On the other hand, when the development time is significantly different from the integration time, one of the teams would have to wait for the completion of the task in the other team. In this context, a typical application area of the proposed framework is the design and development of large-scale components in complex industrial products such as vehicles, vessels, and aircraft.

The rest of this paper is organized as follows. In Section~\ref{sec:model}, we describe our model of PD processes under asynchronous and aperiodic interactions between system integration and local development teams. After studying the feasibility of the proposed PD process model in Section~\ref{sec:analysis}, in Section~\ref{sec:optimization} we develop an optimization framework for efficiently solving the time-cost trade-off problems. The obtained framework is illustrated with a real PD process in Section~\ref{sec:caseStudy}. \textcolor{\changesColor}{Further boundary analyses based on major graph models are performed in Section~\ref{sec:synth}.} We finally provide the conclusion of the paper as well as some discussions in Section~\ref{sec:con}.

\section{Asynchronous and aperiodic work transformation model}\label{sec:model}

In this section, we review the asynchronous work transformation model presented in~\cite{Yassine2003}. We then generalize the model to the case of random and aperiodic information exchange. We finally describe the problems of feasibility analysis and optimal resource allocation studied in this paper.

\subsection{Asynchronous work transformation} \label{subsec:pwt}

In the asynchronous and periodic work transformation model~\cite{Yassine2003}, there exist a pair of \emph{local} and \emph{system} teams working for the development of a product. The PD process contains $m$ tasks, and each task is separated into development and integration tasks that are performed by the local and system teams, respectively. As in~\cite{Yassine2003}, we let $L_i(k)$ and $S_i(k)$ represent the amount of unfinished work in the $i$th local and system task at time step~$k$, respectively. Let us define the vectors
\begin{equation*}
L(k) = \begin{bmatrix}
L_1(k) \\ \vdots \\ L_m(k)
\end{bmatrix},\ 
S(k) = \begin{bmatrix}
S_1(k) \\ \vdots \\ S_m(k)
\end{bmatrix}. 
\end{equation*}
If the local and system teams exchange information at every time step, the above variables evolve over time by the following equation: 
\begin{equation}\label{eq:def:baseLSeq}
\begin{multlined}
\begin{bmatrix}
L(k+1)\\ S(k+1)
\end{bmatrix} = 
\begin{bmatrix}
W_L & W_{SL} 
\\
W_{LS} & W_{S}
\end{bmatrix}
\begin{bmatrix}
L(k)\\ S(k) 
\end{bmatrix}, 
\end{multlined}
\end{equation}
where $W_L$, $W_{SL}$, $W_{LS}$, and $W_{S}$ are WTMs having nonnegative entries. For example, $W_{SL}$ is an entry-wise nonnegative matrix that captures the rework fraction created by system tasks~$S(k)$ for the corresponding local tasks~$L(k)$. The other WTMs are understood in a similar manner.

As discussed and demonstrated by the authors in~\cite{Yassine2003}, even though the local team may frequently provide the system team with local updates, the system team, on the other hand, may provide only intermittent feedback to the local team. To model this situation, they introduce the third variable, $H_i(k)$, denoting the amount of finished work in the $i$th system task that is not yet transfered to the local team. When feedback from the system to local team occurs, the finished work will be cleared on the part of the system team and will be transferred to the local team. Also, until feedback occurs, the finished work keeps accumulating within the system team. In this situation, the vectorial variable
\begin{equation*}
H(k) = \begin{bmatrix}
H_1(k) \\ \vdots \\ H_m(k)
\end{bmatrix} 
\end{equation*}
dynamically evolves as
\begin{equation*}
\begin{multlined}[.95\linewidth]
H(k+1) = 
%\\
\begin{cases}
0,\quad \mbox{if feedback occurs at time~$k$,}
\\
H(k) + W_{SH} S(k),\quad \mbox{otherwise, }
\end{cases}
\end{multlined}
\end{equation*}
where $W_{SH}$ is a WTM from unfinished system tasks to finished system tasks. Similarly, the amount of unfinished work in the local team evolves by the following difference equation:
\begin{equation*}
\begin{multlined}[.95\linewidth]
L(k+1) = 
%\\
\begin{cases}
W_L L(k) + W_{SL} S(k) + H(k),
& 
\\\quad
\mbox{if feedback occurs at time~$k$,}
\\
W_L L(k) ,\quad
%\quad
%&
\mbox{otherwise.}
\end{cases}
\end{multlined}
\end{equation*}
Finally, notice that the amount of unfinished work in the system team evolves in the same way as in~\eqref{eq:def:baseLSeq}. Combining the above equations, we see that the joint state variable
\begin{equation}\label{eq:state}
x(k) = \begin{bmatrix}
L(k)\\ S(k) \\ H(k)
\end{bmatrix}
\end{equation}
evolves over time as 
\begin{equation}\label{eq:def:intermittentLSeq}
\begin{multlined}
x(k+1) = 
\begin{cases}
A_{1}x(k), & \mbox{if feedback occurs at time~$k$,}\\
A_2x(k), & \mbox{otherwise, }
\end{cases}
\end{multlined}
\end{equation}
where the matrices $A_1$ and $A_2$ are given by 
\begin{equation*}
A_1  = \begin{bmatrix}
W_L & W_{SL}  & I
\\
W_{LS} & W_{S} & O
\\
O & O & O
\end{bmatrix}, \ 
A_2 = \begin{bmatrix}
W_L & O & O 
\\
W_{LS} & W_{S} & O 
\\
O & W_{SH} & I
\end{bmatrix}.
\end{equation*}

In~\cite{Yassine2003}, the authors considered the situation where feedbacks occur periodically. By using the Floquet theory for periodic linear dynamical systems, the authors clarified how intermittent feedback from the system team leads to persistent recurrence of PD problems called design churn effects, where progress oscillates between being on schedule and falling behind.

\subsection{Specifications of WTMs}

In this paper, we adopt the specification of WTMs proposed in~\cite{Yassine2003}. Let $\Omega_L$ and $\Omega_S$ denote the DSMs within the local and system teams, respectively. For each $\sigma = L, S$, the diagonal elements of~$\Omega_{\sigma}$ give work completion coefficients, while the off-diagonal elements ($\Omega_{\sigma, ij}$ with $i\neq j$) denote the amount of rework created for the task~$i$ per unit of work done on the task~$j$. Also, for describing the dependency between the local and system teams, inter-component dependency matrices (IDMs) are introduced. Let $\Omega_{LS}$ denote the IDM, whose element~$\Omega_{LS,ij}$ represents the amount of rework created for system task~$i$ per unit of work done on local task~$j$. The IDM~$\Omega_{SL}$ is understood in the same manner.

Then, the authors in~\cite{Yassine2003} have adopted the following formula for computing the WTMs from the DSMs and IDMs. In their formula, the WTM within the local team is given by
\begin{equation}\label{eq:WLij}
W_{L, ij} = 
\begin{cases}
1-\Omega_{L, ii},&\mbox{if $i=j$}, 
\\
\Omega_{L, ij}\Omega_{L, jj},&\mbox{if $i\neq j$}.
\end{cases}
\end{equation}
The system WTM $W_S$ is computed in the same manner. Similarly, the IDMs are given by the formulas
\begin{equation}\label{eq:WLSijWSLij}
\begin{aligned}
W_{LS, ij} &= \Omega_{LS, ij} \Omega_{L, jj}, 
\\
W_{SL, ij} &= \Omega_{SL, ij} \Omega_{S, jj},
\end{aligned}
\end{equation}
for all $1\leq i, j\leq m$. Finally, they let $W_{SH} = W_{SL}$.

\subsection{Aperiodic work transformation}

\begin{figure}
\centering
\includegraphics[width=.8\linewidth]{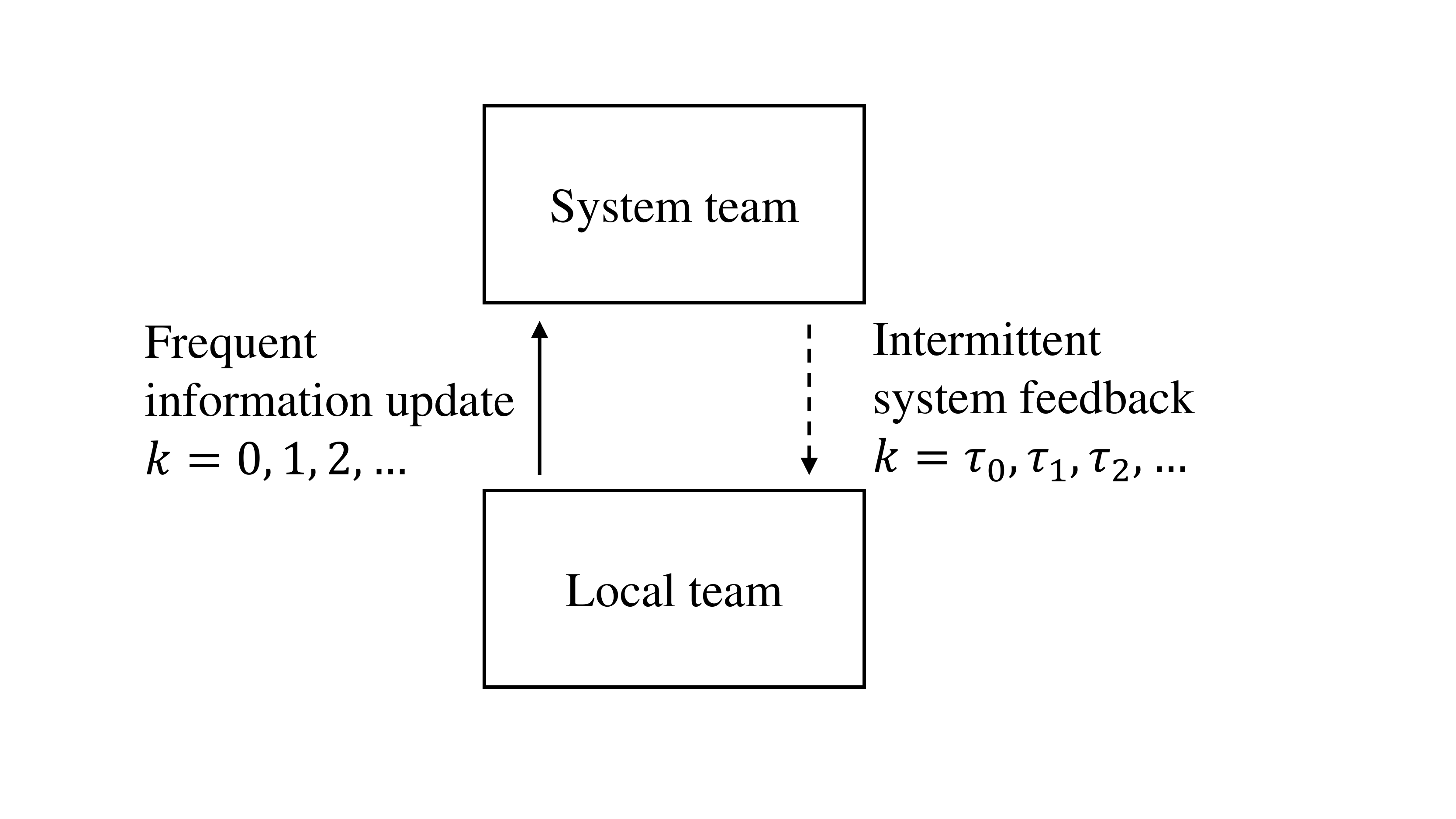}
\caption{{Asynchronous and aperiodic information exchange between the system and local teams.}}
\label{fig:probsetting}
\end{figure}

Let $\tau_0$, $\tau_1$, $\dotsc$ denote the times at which feedback from the system team to the local team occurs (see Fig.~\ref{fig:probsetting} for a schematic picture). Without loss of generality, we assume that $\tau_0 = 0$. Let
\begin{equation*}\label{eq:Deltatauell}
\Delta \tau_{\ell}  = \tau_{\ell+1} - \tau_{\ell}
\end{equation*}
denote the interval between feedbacks. In \cite{Yassine2003}, it was assumed that the interval $\Delta \tau_{\ell}$ is constant. However, in practical PD processes, it is common that the interval of feedbacks fluctuates due to both intrinsic and exogenous factors. For this reason, in this paper, we consider a generalized work transformation model where the interval experiences stochastic fluctuations. Let $h_{\min}$ and $h_{\max}$ denote the minimum and maximum length of interval, i.e., we assume that 
\begin{equation}\label{eq:def:hminhmax}
0< h_{\min} \leq  \Delta \tau_{\ell} \leq h_{\max}. 
\end{equation} 
We furthermore assume that the random intervals~$\Delta \tau_{\ell}$ are independent and identically distributed with the probability
\begin{equation}\label{eq:randomInerval}
P(\Delta \tau_{\ell} = h) = p_h, \ h =1, 2, \dotsc, 
\end{equation}
which can be estimated by the manager from the past history of project management data. We remark that, under this formulation, we can no longer depend on the Floquet theory on periodic linear dynamical systems that was used in~\cite{Yassine2003} because the dynamical system~\eqref{eq:state} is no longer a periodic linear system but a stochastic switched linear system in the context of systems and control theory.

In this paper, we address the following two fundamental questions on the PD process described above. The first question is about the feasibility of the PD process. Specifically, we are interested in if the state variable~$x(t)$ in~\eqref{eq:state} tends to $0$ as $t\to\infty$. This feasibility analysis problem is investigated in Section~\ref{sec:analysis}. The other question is on the PD process optimization. Specifically, if the PD process turns out to be infeasible, or, feasible but too slow, we are interested in how to optimally improve the PD process by utilizing a limited budget. In Section~\ref{sec:optimization}, we study the problem of tuning the dependencies between different product components and development teams for optimally improving the process in a cost-efficient manner.

\color{black}
\subsection{Literature review: Linear systems modeling}

Before developing our framework in the next section, in this section we give a brief review of the linear systems modeling of PD processes and, then, state the contribution of this paper. 

\subsubsection{Static dependency structure}

The linear systems approach for modeling PD processes dates back to the seminal work by Smith and Eppinger in~\cite{Smith1995}, where the authors modeled the amount of unfinished work by a linear system having, as its transition matrix, the WTM of the development process. Based on the linear system model, the authors illustrated the critical role of eigenvalues and eigenvectors of the WTM for understanding the behavior of PD processes. They have also shown that the values of the dominating eigenvector of the WTM allow us to identify the design features that require a large number of iterations. This linear systems modeling was later extended to continuous-time models in terms of differential equations~\cite{Joglekar2001,Ong2003,Kim2007a}. Although a linear system model is an idealization of realistic PD processes, the model indeed allows us to obtain new knowledge to and identify the critical issues in PD processes~\cite{Smith1995}. For this reason, the linear systems modeling has been adopted in other contexts of engineering management including value-based managements~\cite{Hahn2012}, cash flow management~\cite{Cui2010}, and the KPI analysis of supply chains~\cite{Cai2009}.

When the analysis of a PD process indicates that the project is not feasible within a pre-specified period and, therefore, needs to be accelerated, the project manager would consider improving the process by allocating additional human or technological resources. In this direction, we can find in the literature several methodologies for tuning the values of WTMs for process improvements. Lee~et~al.~\cite{Lee2004b} presented an analytical framework for tuning the value of the WTM based on the pole-placement method developed in the systems and control theory. Cheng and Chu~\cite{Cheng2012} extended this approach by using fuzzy numbers and proposed an optimization framework for assigning tasks to multi-skilled employees. Joglekar and Ford~\cite{Joglekar2005} employed a more advanced control theoretical tool called the LQ optimal control, which allows us to find a resource allocation that optimizes the performance of the PD process over an infinite time-horizon. 

\subsubsection{Time-varying dependency architecture}

Huberman and Wilkinson~\cite{Huberman2005a} proposed a general linear system model where the WTM randomly fluctuates over time. The sources of such fluctuations include asynchronicity of information exchange, uncertainty in performance evaluation, and fluctuations in the amount of development resources. Under the assumption that the fluctuations of the WTM are independent and identically distributed, the authors showed that the time-averaged dynamics of the PD process can be analyzed by the eigenvalues and eigenvectors of the time-average of the WTM. The authors also investigated the dependencies of the higher order moments of the work vectors on the parameters of the probability distributions of the fluctuations.

The authors in~\cite{Xiao2011} considered a linear system model, in which reworks of tasks occur randomly. The authors proposed an analytical method, based on the eigendecomposition of the WTM, to mitigate the development process for minimizing the project duration. However, since their method is based on the analysis of the worst-case in which reworks occur with probability one, the mitigation strategy resulting from their methodology is not expected to be cost-effective.

Yassine et al.~\cite{Yassine2003} investigated the effect of asynchronous information exchanges between local development teams and the system integration team in PD processes. The authors specifically modeled the asynchronicity by a time-varying linear system whose WTM changes periodically in accordance with the period of information exchanges. Using the Floquet theory of linear periodic systems~\cite{Richards1983}, the authors have shown that the churn phenomena, which are frequently observed in realistic PD processes, are caused by the asynchronicity. The authors also suggested a heuristic methodology for improving the project development processes based on the eigendecomposition of a certain matrix following from the Floquet theory.

As we have reviewed so far in this subsection, although there exist various analytical methodologies for the improvement and optimization of fluctuation-free PD processes, we still find a lack of such methodologies tailored for the development processes subject to temporal fluctuations. An exception is a work by Yassine et al.~\cite{Yassine2003}, where the authors proposed an eigenvector-based method for identifying design components requiring improvements. However, it remains to be an open question how much resources we should allocate when multiple such components are identified. We also remark that the optimality of resource allocations is rarely discussed in the literature as well.
\color{black}

\section{Feasibility analysis}\label{sec:analysis}

In this section, we analyze the feasibility of the PD process described in the last section. We specifically show that the PD process is feasible if and only if the magnitudes of the eigenvalues of a generalized WTM are strictly less than one.

{We start our analysis by reviewing the results presented by the authors in~\cite{Yassine2003}, in which they have considered the case of periodic feedbacks from the system team to the local team. Let the period of the feedback be denoted by $T>0$. Also, let the dimension of the vector $x$ be denoted by~$n = 3m$. Then, by using the Floquet theory of periodic linear systems, the authors found that the state variable~\eqref{eq:state} admits the representation
\begin{equation}\label{eq:Yassine}
x(k) = \Pi(k) Z^k g,\quad k=0, 1, 2, \dotsc,  
\end{equation}
where $Z \in \mathbb R^{n\times n}$ is a matrix, $g \in \mathbb R^n$ is a  vector, and $\Pi(0)$, $\Pi(1)$, $\Pi(2)$,~\dots~is a sequence of $n\times n$ matrices having period~$T$. By using the representation~\eqref{eq:Yassine}, the authors identified the following two sources of design churn effects: the one arising from the periodicity of the matrix sequence~$\Pi$, and the one from the positions of the eigenvalues of the matrix~$Z$. In their analysis, an important role is played by the monodromy matrix
\begin{equation}\label{eq:monodromy}
A_2^{T-1}A_1, 
\end{equation}
which represents the work transformation between feedback epochs, i.e., the transitions of the sub-sampled state variables 
\begin{equation}\label{eq:defxi:pre}
x(0),\ x(T),\ x(2T), \dotsc .
\end{equation}
In fact, the matrix $Z$ in \eqref{eq:Yassine} is obtained from the diagonalization of the monodromy matrix~\eqref{eq:monodromy}.}

However, in our case where the interval of feedbacks is stochastically modeled, the monodromy matrix~\eqref{eq:monodromy} equals
\begin{equation}\label{eq:randomMono}
M = A_2^{\Delta \tau_{\ell}-1}A_1, 
\end{equation}
which is a random matrix. Due to the stochasticity of the monodromy matrix, we can no longer apply the Floquet theory to our PD process. However, we can still attempt to adopt the methodology based on the monodromy matrix. In order to illustrate the idea,  let us define the random vectors
\begin{equation}\label{eq:def:xi}
\zeta(\ell) = x(\tau_{\ell}),\ \ell = 0, 1, 2, \dotsc. 
\end{equation}
As in \eqref{eq:defxi:pre}, the random vector~\eqref{eq:def:xi} represents the state right before the feedback from the system team. From \eqref{eq:state} and \eqref{eq:randomMono}, we see that the vectors $\zeta(\ell)$ satisfy the relationship
\begin{equation}\label{eq:xiDifferenceEq}
\zeta(\ell+1) = M \zeta(\ell). 
\end{equation}
Being a stochastic difference equation, Equation~\eqref{eq:xiDifferenceEq} is still not easy to analyze. In order to avoid the difficulty, let us take the mathematical expectation on both sides of the equation~\eqref{eq:xiDifferenceEq} to obtain
\begin{equation}\label{eq:xiNext}
E[\zeta(\ell+1)] = E[A_2^{\Delta \tau_{\ell}-1}A_1   \zeta(\ell)]. 
\end{equation}
Since the feedback interval~$\Delta \tau_{\ell}$ is drawn from a distribution independent of the state variable~$\zeta(\ell)$, we can decompose the expectation on the right-hand side of \eqref{eq:xiNext} to obtain
\begin{equation}\label{eq:xiNextNext}
E[\zeta(\ell+1)] =  E[A_2^{\Delta \tau_{\ell}-1}]  A_1 E[\zeta(\ell)]. 
\end{equation}
Let us introduce the vector 
\begin{equation*}
z(\ell) = E[\zeta(\ell)]. 
\end{equation*}
Then, from \eqref{eq:xiNextNext}, we obtain 
\begin{equation}\label{eq:xiNextNextNext}
z(\ell+1) = \mathcal M z(\ell),  
\end{equation}
where the matrix~$\mathcal M $ is given by $\mathcal M  = E[A_2^{\Delta \tau_{\ell}-1}]  A_1$. Notice that, by \eqref{eq:randomInerval}, we can calculate the matrix~$\mathcal M$ by the formula
\begin{equation}\label{eq:complexMat}
\mathcal M = (p_{h_{\min}}A_2^{h_{\min}-1} + \cdots + p_{h_{\max}}A_2^{h_{\max}-1}) A_1, 
\end{equation}
where $h_{\min}$ and $h_{\max}$ are the minimum and maximum of the interval of two subsequent feedbacks and satisfy \eqref{eq:def:hminhmax}. 

We can interpret equation \eqref{eq:xiNextNextNext} as representing the expected transformation of works between subsequent feedback epochs, with $\mathcal M$ representing a generalized WTM. As is well-known, if the magnitudes of the eigenvalues of the matrix~$\mathcal M$ are all less than one, then $z(\ell)$, the expected state variable at the feedback epochs, converges to zero as $\ell$ tends to infinity. Notice that, due to the stochasticity of the feedback intervals, this convergence does not necessarily imply the convergence of the original state variable~$x(k)$ as $k$ tends to infinity. However, by utilizing the theory of switched linear systems~\cite{Ogura2014d} in the context of systems and control theory, we can prove that the condition on the magnitude of the eigenvalues of the generalized WTM is, in fact, sufficient for the feasibility of the PD process:

\begin{theorem}\label{thm:main}
The expected amounts of unfinished works, i.e., $E[L(k)]$ and $E[S(k)]$, converge to $0$ as $k$ tends to $\infty$ if and only if the magnitudes of the eigenvalues of~$\mathcal M$ are less than one.
\end{theorem}

\begin{proof}
Define the stochastic process $\xi= \{\xi(k)  \}_{k\geq 0}$ by
\begin{equation*}
\xi(k) = \begin{cases}
1, & \mbox{if feedback occurs at time~$k$}, 
\\
2, & \mbox{otherwise}. 
\end{cases}
\end{equation*}
By the assumption that the intervals~$\Delta \tau_{\ell}$ are independent and identically distributed, we can show that the stochastic process~$\xi$ is regenerative (see~\cite{Smith1955} for the details). Moreover, under this notation, the PD dynamics~\eqref{eq:def:intermittentLSeq} admits the  representation~$x(k+1) = A_{\xi(k)} x(k)$ and, therefore, is a regenerative switched linear system defined in~\cite{Ogura2014d}. Moreover, we can easily confirm that the dynamics \eqref{eq:def:intermittentLSeq} satisfies the three conditions A1, A2, and A3 given in~\cite{Ogura2014d}. The dynamics satisfies A1 because the WTMs are nonnegative. Also, the conditions~A2 and~A3 are satisfied because $h_{\max}$ is finite. Hence, Theorem~25 in~\cite{Ogura2014d} shows that Theorem~\ref{thm:main} holds true.
\end{proof}

\textcolor{black}{By Theorem~\ref{thm:main},  the maximum magnitude of the eigenvalues of the matrix~$\mathcal M$, i.e., the spectral radius~$\rho(\mathcal M)$, determines the feasibility of the PD project. Therefore, the quantity serves as an indicator for the PD manager to quickly predict the final consequence of his or her PD project. If the quantity is less than one, then the manager is assured about the validity of the current resource allocation in design rules. On the other hand, if the quantity is larger than one, then the manager should be warned and urged to reconsider the current practice. For this reason, we call the spectral radius the \emph{feasibility index} of the PD project. In the next section, we present an analytical framework for cost-efficiently optimizing the feasibility index.}

\section{Optimizing dependencies}\label{sec:optimization}

Based on the feasibility analysis in the last section, in this section, we study the problem of tuning the dependencies between distinct design components or teams to improve a nominal, possibly infeasible, PD process. We specifically consider the following PD optimization problem. Suppose that we can use resources for decreasing the dependencies. Assuming that the resources have an associated cost and that we are given a fixed budget, how should we distribute our resource throughout the local and system tasks to accelerate the PD process?

The purpose of this section is to present a resource allocation framework for tuning the dependency terms in the DSMs and IDMs to optimize the PD process under aperiodic interactions between the system integration and local development teams. We show that, under the assumption that the cost functions for tuning the dependencies are posynomials, we can transform the resource allocation problem to an equivalent convex optimization problem.

\subsection{Problem formulation}

In this paper, we shall focus on tuning the values of the inter-component or inter-team dependencies
\begin{equation}\label{eq:dependencies}
\{ \Omega_{\sigma, ij}: (\sigma \in \{LS, SL\}) \mbox{ or } (i\neq j \aand \sigma\in \{L, S\})  \}.
\end{equation}
We assume that there is a cost associated with tuning the strength of dependencies. Let $f_{L, ij}(\Psi_{L, ij})$ denote the cost of tuning the nominal dependency~$\Omega_{L, ij}$ to $\Psi_{L, ij}$. In other words, if we want to set the dependency of the $i$th local component on the $j$th local component to be $\Psi_{L, ij}$, we need to pay $f_{L, ij}(\Psi_{L, ij})$ monetary units. Since we do not need to consider improving the dependencies that are originally zero, the total cost for the improvement of the DSM~$\Omega_L$ equals
\begin{equation*}
C_{L} = \sum_{\substack{(i, j)\colon \Omega_{L, ij}\neq 0 \\{\aand} i\neq j}} f_{L, ij}(\Psi_{L, ij}). 
\end{equation*}
Similarly, we introduce the cost functions~$f_{LS, ij}(\Psi_{LS, ij})$, $f_{SL, ij}(\Psi_{SL, ij})$, and $f_{S, ij}(\Psi_{S, ij})$ for tuning the nominal DSM~$\Omega_{S}$ and IDMs~$\Omega_{LS}$ and  $\Omega_{SL}$, respectively. Then, the total cost for the improvement of the DSM and IDMs are given by
\begin{equation*}
\begin{aligned}
C_{S} &= \sum_{\substack{(i, j)\colon \Omega_{S, ij}\neq 0 \\{\aand} i\neq j}}
f_{S, ij}(\Psi_{S, ij}), 
\\
C_{LS} &= \sum_{(i, j)\colon\Omega_{LS, ij}\neq 0} f_{LS, ij}(\Psi_{LS, ij}), 
\\
C_{SL} &= \sum_{(i, j)\colon\Omega_{SL, ij}\neq 0} f_{SL, ij}(\Psi_{SL, ij}). 
\end{aligned}
\end{equation*}
Thus, the total cost for the improvement of the entire development process equals 
\begin{equation}\label{eq:totalCost}
C = C_L + C_S + C_{LS} + C_{SL}. 
\end{equation}

We impose the following two natural restrictions on the new DSMs and IDMs. First, the new matrices must be ``better'' than the nominal ones. In other words, we impose the following inequalities
\begin{equation}\label{eq:Psi:upper}
\Psi_{\sigma,ij} \leq \Omega_{\sigma,ij}, 
\end{equation} 
where $\sigma$ denotes any one of the strings $L$, $S$, $LS$, and $SL$. Secondly, we assume that there is a certain management limitation in improving the values of the matrices, which we model by the inequality
\begin{equation}\label{eq:Psi:lower}
\Psi_{\sigma,ij} \geq \epsilon \Omega_{\sigma,ij}, 
\end{equation} 
where $\epsilon \in (0, 1)$ is a constant dependent on projects. This inequality implies that the possible improvement of the values of the matrices is at most $100(1-\epsilon)$\%. This implies that, since $\epsilon > 0$, we cannot completely eliminate a dependency that is present in the nominal PD process. 

From our feasibility analysis in the last section, we know that the smaller the feasibility index, the faster the amounts of unfinished work decrease. For this reason, we formulate our first project optimization problem as follows: 

\begin{problem}[Budget-constrained dependency optimization]\label{prb:}
Given cost functions $f_{\sigma, ij}$, a constant $\epsilon > 0$, and a budget~$B>0$, find the new DSMs and IDMs $\Psi_\sigma$ that minimize the feasibility index of the PD project while satisfying the budget constraint  on the cost \eqref{eq:totalCost}
\begin{equation}\label{eq:buggetConst:DSM}
C \leq B
\end{equation}
and the constraints \eqref{eq:Psi:upper} and \eqref{eq:Psi:lower}. 
\end{problem}

We formulate another PD optimization problem of practical interest. In the budget-constrained dependency optimization problem, we need to distribute our resource in the PD process while keeping the corresponding cost within a given available budget. However, when the main KPI is in reducing the improvement cost, the manager would be interested in directly minimizing the cost, $C$, while making sure that the PD process is feasible. From this perspective, we formulate an alternative optimization problem as follows:

\begin{problem}[Performance-constrained dependency optimization]\label{prb:performance:DSM}
Given cost functions $f_{\sigma, ij}$, a constant $\epsilon > 0$, and a required performance $r \in [0, 1)$, find the new DSMs and IDMs that minimize the cost $C$  \eqref{eq:totalCost}  while satisfying  the constraints \eqref{eq:Psi:upper} and \eqref{eq:Psi:lower} as well as the following requirement on the feasibility index: 
\begin{equation}\label{eq:performanceConst}
\rho(\mathcal M) \leq r. 
\end{equation}
\end{problem}

Mathematically, the budget-constrained dependency optimization problem is formulated as 
\begin{equation}\label{eq:origOpt}
\begin{aligned}
\minimize_{\Psi_{\sigma, ij}}\ \ \ \ \,    & \rho(\mathcal M )
\\
\subjectto\ \ \ \ & \eqref{eq:Psi:upper},\,\eqref{eq:Psi:lower},\, \eqref{eq:buggetConst:DSM}. 
\end{aligned}
\end{equation}
Also, the performance-constrained dependency optimization problem is formulated as 
\begin{equation}\label{eq:performanceOPT}
\begin{aligned}
\minimize_{\Psi_{\sigma, ij}}\ \ \ \ \,& C
\\
\subjectto\ \ \ \ & \eqref{eq:Psi:upper},\,\eqref{eq:Psi:lower},\, \eqref{eq:performanceConst}. 
\end{aligned}
\end{equation}

\subsection{Transformation to convex optimizations}

The optimization problems~\eqref{eq:origOpt} and~\eqref{eq:performanceOPT} are not trivial to solve because the problems involve the spectral radius of the complicated matrix sum~$\mathcal M$ of the form~\eqref{eq:complexMat}. Although one may apply a heuristic optimization techniques such as the gradient descent method, such an approach can result in a solution that is only locally optimal, yielding a not necessarily effective allocation of resources. Therefore, it is desirable to establish a reliable method for obtaining a globally optimal solution to the optimization problems~\eqref{eq:origOpt} and~\eqref{eq:performanceOPT}. 

In this subsection, we shall show that the optimization problems can be efficiently solved if the cost functions~$f_{\sigma, ij}$ belong to a wide class of functions called \emph{posynomials} (see, e.g.,~\cite{Boyd2007}), which we review below. Let $f(X)$ be a real function defined for a positive scalar variable~$X>0$. We say that $f$ is a \emph{monomial} if there exist $c>0$ and a real number~$a$ such that $f(X) = cX^{a}$. We say that $f$ is a \emph{posynomial} if $f$ is a sum of monomials. For example, $f(X) = X^{-2} + 2 + 3X^{1.5}$ is a posynomial because $X^{-2}$, $2$, and $3X^{1.5}$ are monomials. The following lemma shows that posynomials exhibit a numerically good property with respect to the auxiliary variable
\begin{equation}\label{eq:logV}
Y = \log X. 
\end{equation}

\begin{lemma}[\cite{Boyd2007}]\label{lem:llconvexity}
If the function~$f(X)$ is a posynomial, then the function~$F$ defined by $F(Y) = \log f(\exp Y)$ is a convex function of~$Y$.
\end{lemma}

Besides the above mentioned numerical properties, there are developed algorithms~\cite{Boyd2007} to fit posynomials to real data, possibly taken from past management histories. For these reasons, in this paper, we place the following reasonable assumption on the cost functions. We assume that, for each $\sigma$, $i$, and $j$, there exists a posynomial $f^+_{\sigma, ij}$ such that the cost function~$f_{\sigma, ij}$ is of the form
\begin{equation*}
f_{\sigma, ij}(\Psi_{\sigma, ij}) = f^+_{\sigma, ij}(\Psi_{\sigma, ij}) - f^+_{\sigma, ij}(\Omega_{\sigma, ij}). 
\end{equation*}
The essential part of the cost function is the first term $f^+_{\sigma, ij}(\Psi_{\sigma, ij})$, while the second term~$-f^+_{\sigma, ij}(\Omega_{\sigma, ij})$ is for normalizing the cost function as $f_{\sigma, ij}(\Omega_{\sigma, ij}) = 0$, i.e., the zero investment yields the nominal interdependency matrix. Corresponding to the decomposition, let us define
\begin{equation*}
\begin{aligned}
C^+_{L} &= \sum_{\substack{(i, j)\colon \Omega_{L, ij}\neq 0 \\{\aand} i\neq j}} f^+_{L, ij}(\Psi_{L, ij}), 
\\
C^+_{S} &= \sum_{\substack{(i, j)\colon \Omega_{S, ij}\neq 0 \\{\aand} i\neq j}}
f^+_{S, ij}(\Psi_{S, ij}), 
\\
C^+_{LS} &= \sum_{(i, j)\colon\Omega_{LS, ij}\neq 0} f^+_{LS, ij}(\Psi_{LS, ij}), 
\\
C^+_{SL} &= \sum_{(i, j)\colon\Omega_{SL, ij}\neq 0} f^+_{SL, ij}(\Psi_{SL, ij})
\end{aligned}
\end{equation*}
and 
\begin{equation}\label{eq:def:Csigma-s}
\begin{aligned}
C^-_{L} &= \sum_{\substack{(i, j)\colon \Omega_{L, ij}\neq 0 \\{\aand} i\neq j}} f^+_{L, ij}(\Omega_{L, ij}), 
\\
C^-_{S} &= \sum_{\substack{(i, j)\colon \Omega_{S, ij}\neq 0 \\{\aand} i\neq j}}
f^+_{S, ij}(\Omega_{S, ij}), 
\\
C^-_{LS} &= \sum_{(i, j)\colon\Omega_{LS, ij}\neq 0} f^+_{LS, ij}(\Omega_{LS, ij}), 
\\
C^-_{SL} &= \sum_{(i, j)\colon\Omega_{SL, ij}\neq 0} f^+_{SL, ij}(\Omega_{SL, ij}). 
\end{aligned}
\end{equation}
Let
\begin{align}
C^+ &= C^+_L + C^+_S + C^+_{LS} + C^+_{SL}, \notag
\\
C^- &= C^-_L + C^-_S + C^-_{LS} + C^-_{SL}. \label{eq:def:C-}
\end{align}
Then, the total cost~\eqref{eq:totalCost} is rewritten as 
\begin{equation*}
C = C^+ - C^-. 
\end{equation*}

We are now ready to state the first main result of this paper. As in \eqref{eq:logV}, we introduce the auxiliary variable
\begin{equation*}\label{eq:transformation}
\Xi_{\sigma, ij} = \log \Psi_{\sigma,ij}
\end{equation*}
for all $\sigma$, $i$, and $j$. By using the celebrated convexity result of Kingman~\cite{Kingman1961}, we can show that the optimization problem~\eqref{eq:origOpt} for solving the budget-constrained dependency optimization problem can be transformed to a convex optimization problem, which can be efficiently and optimally solved by off-the-shelve softwares.

\begin{theorem}\label{thm:convexity:IDMpc}
Let $\Xi_{\sigma, ij}^\star$ be the solution of the optimization problem
\begin{subequations}\label{eq:origOptDMSIDM:pc}
\begin{align}
\minimize_{\Xi_{\sigma, ij }}\ \ \ \ \,& \log \rho(\mathcal M) \notag
\\
\subjectto\ \ \ \ 
&\Xi_{\sigma, ij} \leq \log \Omega_{\sigma, ij}, \label{eq:constupper}
\\
& \Xi_{\sigma, ij} \geq \log \epsilon + \log \Omega_{\sigma, ij}, \label{eq:constlower}
\\
& \log C^+\leq \log (B + C^-). \label{eq:constCost}
\end{align}
\end{subequations}
Then, the DSMs and IDMs defined by
\begin{equation}\label{eq:psiiotastar}
\Psi_{\sigma, ij}^\star = 
\begin{cases}
0,&\mbox{if $\Omega_{\sigma, ij} = 0$, }
\\
\exp(\Xi_{\sigma, ij}^\star),& \mbox{otherwise}, 
\end{cases}
\end{equation}
solve the budget-constrained dependency optimization problem. Moreover, if the cost functions $f_{\sigma, ij}$ are posynomial, then the optimization problem \eqref{eq:origOptDMSIDM:pc} is convex. 
\end{theorem}

\color{black}Theorem~\ref{thm:convexity:IDMpc} serves as an analytical decision support tool for PD managers. Specifically, the theorem allows PD managers to efficiently solve the budget-constrained dependency optimization problem (Problem~\ref{prb:}) by using off-the-shelve software for convex optimization such as \texttt{fmincon} routine in the MATLAB programming language or the \texttt{CVXOPT} package for the Python programming language. This feature distinguishes the proposed framework from others relying on heuristic or non-convex optimization procedures~\cite{Peteghem2010,Alcaraz2003,Boctor1996,Patterson1990} for the following two reasons. On the one hand, since the above solvers always provide  globally optimal solutions, PD managers do not need to worry about the optimality of their investments. On the other hand, since the computational cost for solving the problem is relatively small (polynomial in the size of the project), PD managers can quickly assess and, if necessary, re-design their PD process even when their project is quite large. 

In order to prove Theorem~\ref{thm:convexity:IDMpc}, we state the following celebrated result by Kingman~\cite{Kingman1961}: 

\begin{proposition}[{\cite{Kingman1961}}]\label{prop:Kingman}
Let $\mathcal M$ be a square matrix function in positive variables $x_1$, \dots, $x_N$. Assume that the logarithm of each entry in the matrix~$\mathcal M$ is convex in the variables $x_1$, \dots, $x_N$. Then, $\log \rho(\mathcal M)$, as a function of the variables $x_1$, \dots, $x_N$, is convex. 
\end{proposition}

We can now prove Theorem~\ref{thm:convexity:IDMpc}: \color{black}

\begin{proof}
It is easy to see that the dependency constraints \eqref{eq:Psi:upper}, \eqref{eq:Psi:lower} and the budget constraint~\eqref{eq:buggetConst:DSM} in the optimization problem~\eqref{eq:origOpt} are equivalent to the constraints \eqref{eq:constupper}, \eqref{eq:constlower}, and \eqref{eq:constCost}, respectively. 

Therefore, to prove Theorem~\ref{thm:convexity:IDMpc}, it is enough to show that the quantities~$\log C^+$ and~$\log \rho(\mathcal M)$ are convex with respect to the variables~$\Xi_{\sigma, ij}$ ($\sigma \in \{L, S, LS, SL\}$ and $1\leq i, j \leq m$). To show the convexity of~$\log C^+$, notice that $C^+$ is a posynomial in the dependency variables 
\begin{equation}\label{eq:dependecnyVars}
\{ \Psi_{\sigma, ij}: (\sigma \in \{LS, SL\}) \mbox{ or } (i\neq j \aand \sigma\in \{L, S\})  \}
\end{equation}
by our assumption that the cost function~$f^+_{\sigma, ij}(\Psi_{\sigma, ij})$ is a posynomial for all $\sigma$, $i$, and $j$. Therefore, Lemma~\ref{lem:llconvexity} shows the convexity of~$\log C^+$ with respect to the variables~$\Xi_{\sigma, ij}$. 

Let us show the convexity of~$\log \rho(\mathcal M)$. By equation~\eqref{eq:complexMat}, each entry of the matrix~$\mathcal M$ is a posynomial in the elements of the WTMs. Moreover, by \eqref{eq:WLij} and \eqref{eq:WLSijWSLij}, each element of the WTMs is a posynomial in the dependency variables~\eqref{eq:dependecnyVars}. Therefore, since a composition of posynomials is a posynomial, each entry of the matrix~$\mathcal M$ is a posynomial in the variables~\eqref{eq:dependecnyVars}. Therefore, Lemma~\ref{lem:llconvexity} shows that the logarithm of each entry of~$\mathcal M$ is convex with respect to the variables~$\Xi_{\sigma, ij}$. Hence, Proposition~\ref{prop:Kingman} concludes the convexity of~$\log \rho(\mathcal M)$ with respect to the variables~$\Xi_{\sigma, ij}$, as desired.
\end{proof}

Our second main result states that the performance-constrained dependency optimization problem can be also transformed to a convex optimization problem:

\begin{theorem}\label{thm:convexity:IDMbc}
Let $\Xi_{\sigma, ij}^\star$ be the solution of the optimization problem 
\begin{equation}\label{eq:origOptDMSIDM:bc}
\begin{aligned}
\minimize_{\Xi_{\sigma, ij}}\ \ \ \ \,& \log C^+
\\
\subjectto\ \ \ \ 
&\Xi_{\sigma, ij} \leq \log \Omega_{\sigma, ij}, 
\\
& \Xi_{\sigma, ij} \geq \log \epsilon + \log \Omega_{\sigma, ij}, 
\\
& \log \rho(\mathcal M) \leq \log r. 
\end{aligned}
\end{equation}
Then, the DSMs and IDMs defined by \eqref{eq:psiiotastar} solve the performance-constrained dependency optimization problem. Moreover, if the cost functions $f_{\sigma, ij}$ are posynomial, then the optimization problem \eqref{eq:origOptDMSIDM:bc} is a convex program.
\end{theorem}

\begin{proof}
We can prove Theorem~\ref{thm:convexity:IDMbc} in the same way as in the proof of Theorem~\ref{thm:convexity:IDMpc}. The constraints~\eqref{eq:Psi:upper}, \eqref{eq:Psi:lower}, and~\eqref{eq:performanceConst} are obviously equivalent to the constraints in the optimization problem~\eqref{eq:origOptDMSIDM:bc}. Also, since $C^-$ is given by equations~\eqref{eq:def:Csigma-s} and~\eqref{eq:def:C-} and, therefore, is a constant, the minimization of $C$ performed in the optimization problem~\eqref{eq:performanceOPT} is equivalent to minimizing $\log C^+$ as in the optimization problem~\eqref{eq:origOptDMSIDM:bc}. This completes the proof of Theorem~\ref{thm:convexity:IDMbc}.
\end{proof}

\section{Case study: automotive appearance design}\label{sec:caseStudy}

In this section, we illustrate our optimization framework presented in the last section by using a real PD process reported in~\cite{McDaniel1996}. In Section~\ref{sec:overview}, we give an overview of the nominal automotive appearance design process. In Section~\ref{sec:resource}, we introduce the cost function used in this case study and also give the baseline resource allocation strategy based on the description in~\cite{Yassine2003}. We illustrate the effectiveness of our optimization framework in Section~\ref{sec:results}.

\begin{figure*}[tb]
\centering
\includegraphics[width=.95\tablewidth]{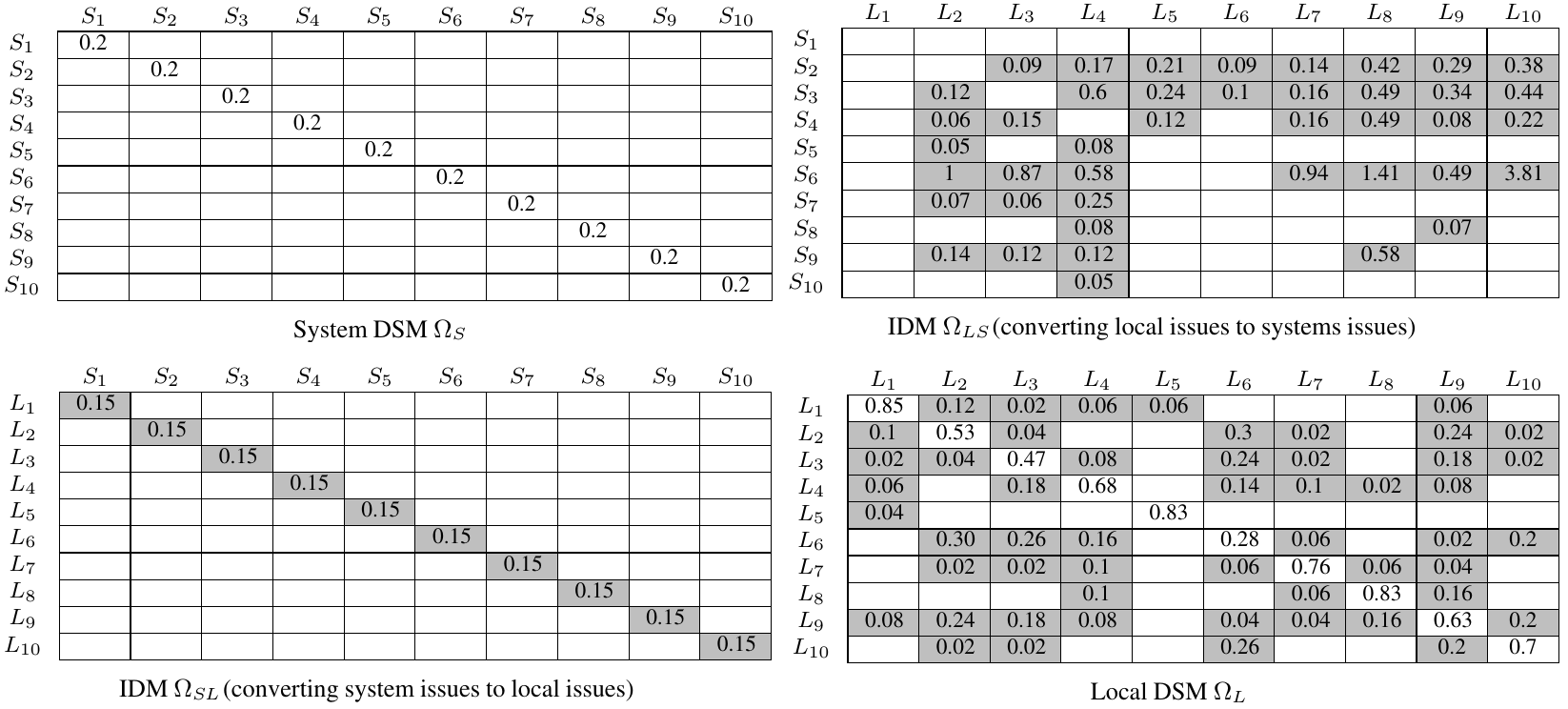}
\caption{{The nominal DSMs and IDMs of the automotive appearance design. The inter-component and inter-team dependencies \eqref{eq:dependencies} that can be weakened by the manager are highlighted with the gray color.}}
\label{table:DSMIDM}
\end{figure*}

\subsection{Automotive appearance design}\label{sec:overview}

The automotive appearance design process reported in \cite{McDaniel1996} (and further investigated in~\cite{Yassine2003}) is a part of an automobile PD process and refers to the process of designing all interior and exterior automobile surfaces for better appearance, surface quality, and operational interface. The authors in~\cite{Yassine2003} focused on the following pair of the system and local teams; the engineering (local) team responsible for the feasibility of designs, and the styling (system) team responsible for the appearance of the vehicle. Information exchanges occur not only on the cross-functional level but also within functional groups working with specific tasks on appearance design. The tasks are 1) carpet, 2) center console, 3) door trim panel, 4) garnish trim, 5) overhead system, 6) instrument panel, 7) luggage trim, 8) package tray, 9) seats, and 10) steering wheel. For the sake of completeness, we include the values of the DSMs and IDMs in Fig.~\ref{table:DSMIDM}. 

During the project period, there occur two different types of information exchanges between the teams. One is a weekly feasibility meeting, where the engineering team feedbacks to the styling team on infeasible design conditions. The other ones are in terms of CAD data from the styling team to the engineering team and are scheduled to be roughly six-week intervals. In this paper, we consider the situation where the schedule of the meeting can be either brought forward or postponed at most two weeks due to random and unexpected circumstances. In order to realize this situation, we set the minimum and the maximum interval of feedbacks as $h_{\min} = 4$ and $h_{\max} = 8$.

\begin{figure}[tb]
\centering
\includegraphics[height=\figheight]{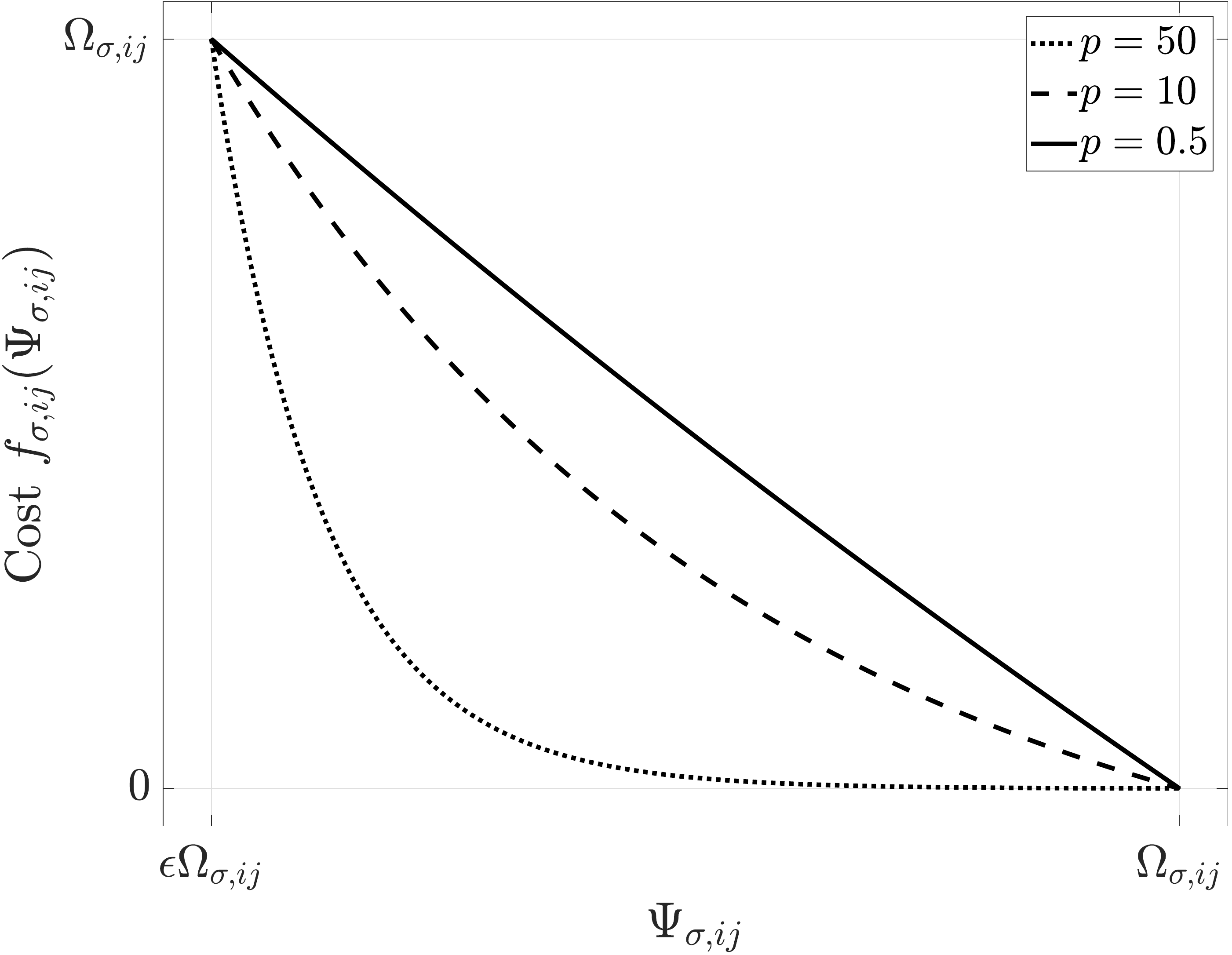}
\caption{Cost function \eqref{eq:costFunction} with $\epsilon = 0.85$ and $p = 50$, $10$, and $0.5$. No investment (i.e., $f_{\sigma, ij}(\Psi_{\sigma, ij}) = 0$) yields the nominal dependency $\Omega_{\sigma, ij}$. The maximum investment ($\Omega_{\sigma, ij}$ monetary unit) yields the fully improved dependency $\epsilon \Omega_{\sigma, ij}$. }
\label{fig:cost}
\end{figure}

\subsection{Resource allocation strategies}\label{sec:resource}

In the nominal DSMs and IDMs shown in Fig.~\ref{table:DSMIDM}, we identify 104 dependencies of the form~\eqref{eq:dependencies}. Since the DSM~$\Omega_S$ is diagonal and, therefore, does not contain dependency terms, we do not consider tuning the values of~$\Omega_S$ in this case study.

We consider the following requirements on the cost functions in this case study. 
\begin{enumerate}
\item The cost function~$f_{\sigma, ij}$ is decreasing, that is, the more we invest, the weaker dependencies become. 
\item In order to achieve the full improvement and set the dependency to $\epsilon\Omega_{\sigma, ij}$, we need to pay the cost~$\Omega_{\sigma, ij}$, that is, %\setcounter{equation}{27}
\begin{equation}\label{eq:fullInvestment}
f_{\sigma, ij}(\epsilon\Omega_{\sigma, ij}) = \Omega_{\sigma, ij}. 
\end{equation}
\end{enumerate}
The second requirement in particular implies that the cost for the full improvement is proportional to the strength of the nominal dependencies. In order to realize this situation, we use the following cost function: 
\begin{equation}\label{eq:costFunction}
f_{\sigma, ij}(\Psi_{\sigma, ij}) = c_{\sigma, ij}\left(
\frac{1}{(\Psi_{\sigma, ij})^{p}} - \frac{1}{({\Omega\mathstrut}_{\sigma, ij})^{p}}
\right), 
\end{equation} 
where $p$ is a parameter that should be appropriately set by the manager, and the constant~$c_{\sigma, ij}$ is uniquely determined to satisfy the equation~\eqref{eq:fullInvestment}. Also, we use a positive $p$ to satisfy the first requirement. In Fig.~\ref{fig:cost}, we show the shape of the cost function for various values of~$p$. {We see that, by tuning the value of $p$, we can recover the following two practical improvement regimes that appear in realistic PD process improvements. When $p$ is small, we obtain  the regime of linear improvement where the magnitude of the improvement, $\Omega_{\sigma, ij} - \Psi_{\sigma, ij}$, is proportional to the allocated resource. On the other hand, a large $p$ yields the regime of diminishing returns (see, e.g., \cite{Chen2012}), where the incremental elimination of the dependency decreases as we allocate more resource.}

{To set the baseline strategy for resource distribution, we recall the strategy that was found to be most effective among the other strategies that were proposed by the authors in~\cite{Yassine2003}. In this reference, the tasks~$L_2$, $L_3$, and $L_6$ (i.e., the tasks for the center console, door trip panel, and instrument panel within the local team) were identified as ``complex local components''. Then, the authors proposed that the project manager should focus on weakening the dependencies between these three local tasks and other tasks. Therefore, in terms of the DSMs and IDMs, it was recommended that we weaken the following dependencies:
\begin{equation}\label{eq:shouldbe}
\begin{multlined}
\Omega_{L, i2},\,\Omega_{L, i3},\,\Omega_{L, i6},\,\Omega_{L, 2j},\,\Omega_{L, 3j},\,\Omega_{L, 6j},\\
\Omega_{LS, i2},\,\Omega_{LS, i3},\,\Omega_{LS, i6},\,\Omega_{SL, 2j},\, \Omega_{SL, 3j},\,\Omega_{SL, 6j}, 
\end{multlined}
\end{equation}
where $i, j = 1, \dotsc, 10$. We choose to implement this recommendation in the following manner. In our implementation, we distribute the whole resource to the dependencies in \eqref{eq:shouldbe} in such a way that the distributed resources are proportional to the strength of the nominal dependencies. Hence, in the baseline strategy, we distribute no resources to the dependencies that are not listed in \eqref{eq:shouldbe}.}

\begin{figure*}[tb]
\begin{minipage}[t]{.475\linewidth}
\centering
\includegraphics[clip,trim=2.42in 2.709in 2.343in 2.812in,width=\tablewidth]{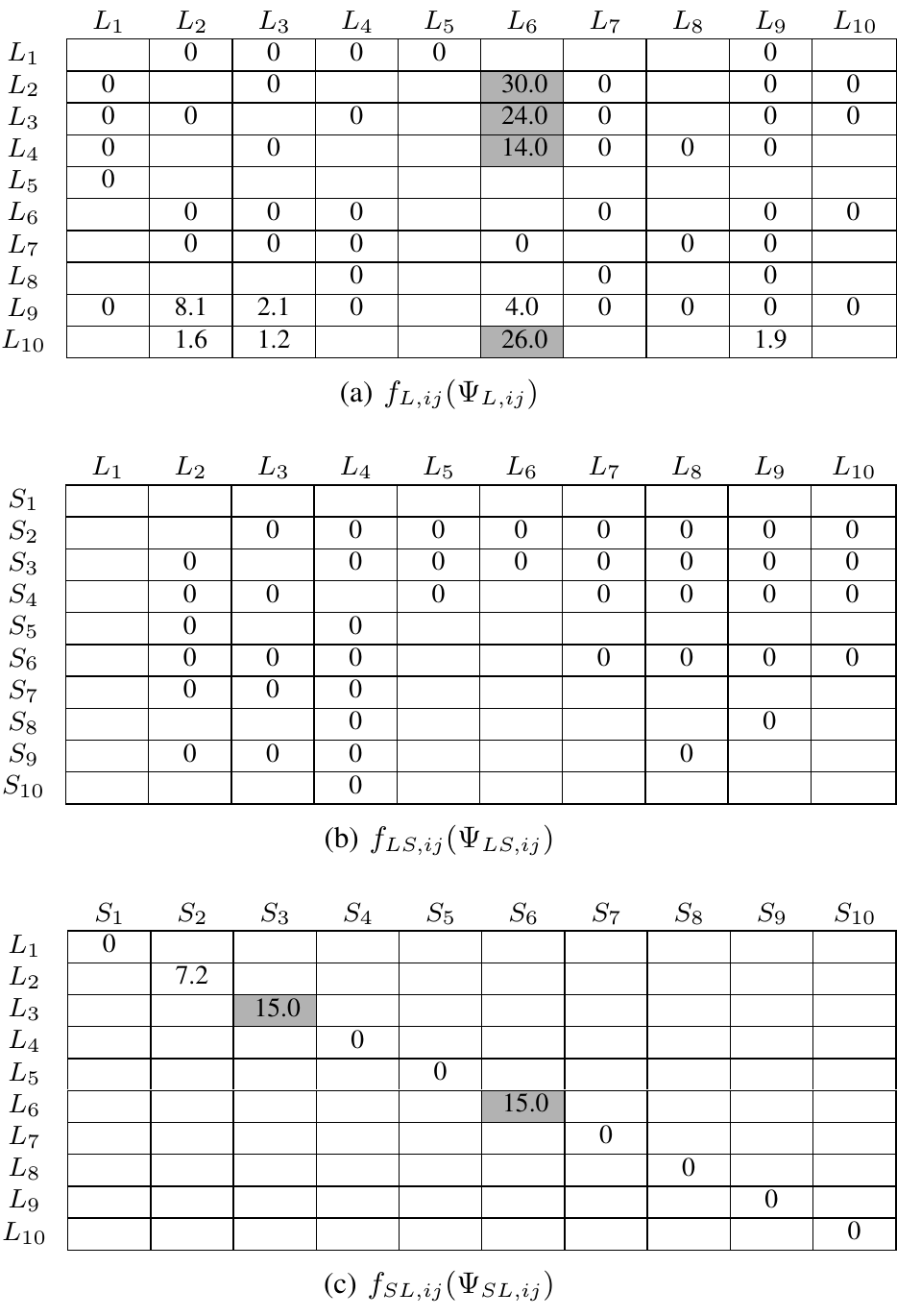}
\caption{{Proposed investments on dependencies (multiplied by 100). The top six investments are highlighted with the gray color.}}
\label{table:optInvestments}
\end{minipage}
\hfil
\begin{minipage}[t]{.475\linewidth}
\centering
\includegraphics[clip,trim=2.42in 2.709in 2.343in 2.812in,width=\tablewidth]{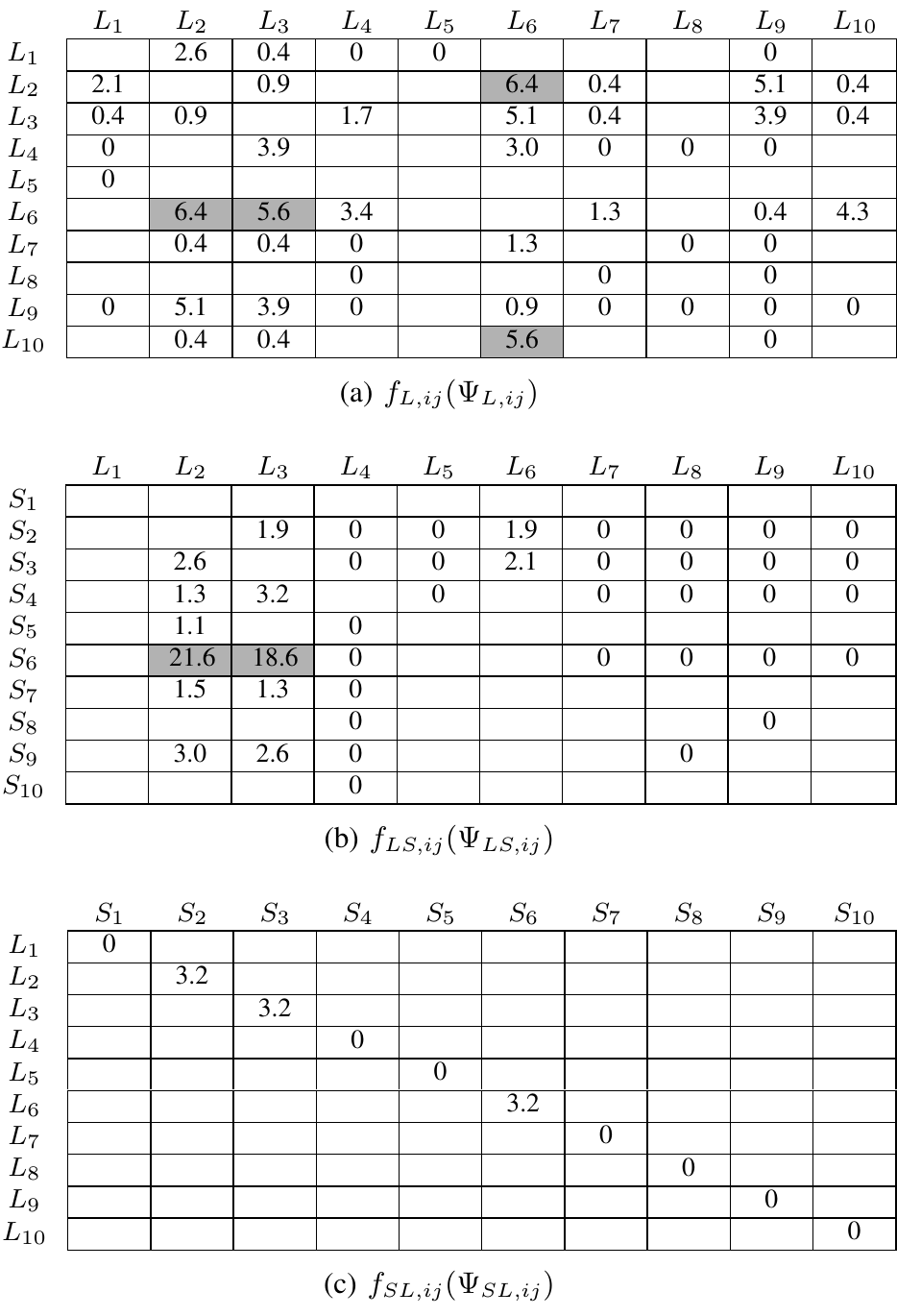}
\caption{{Baseline investments on dependencies (multiplied by 100). The top six investments are highlighted with the gray color.}}
\vspace{3mm}
\label{table:conInvestments}
\end{minipage}
\color{\changesColor}
\centering
\includegraphics[width=.9\linewidth]{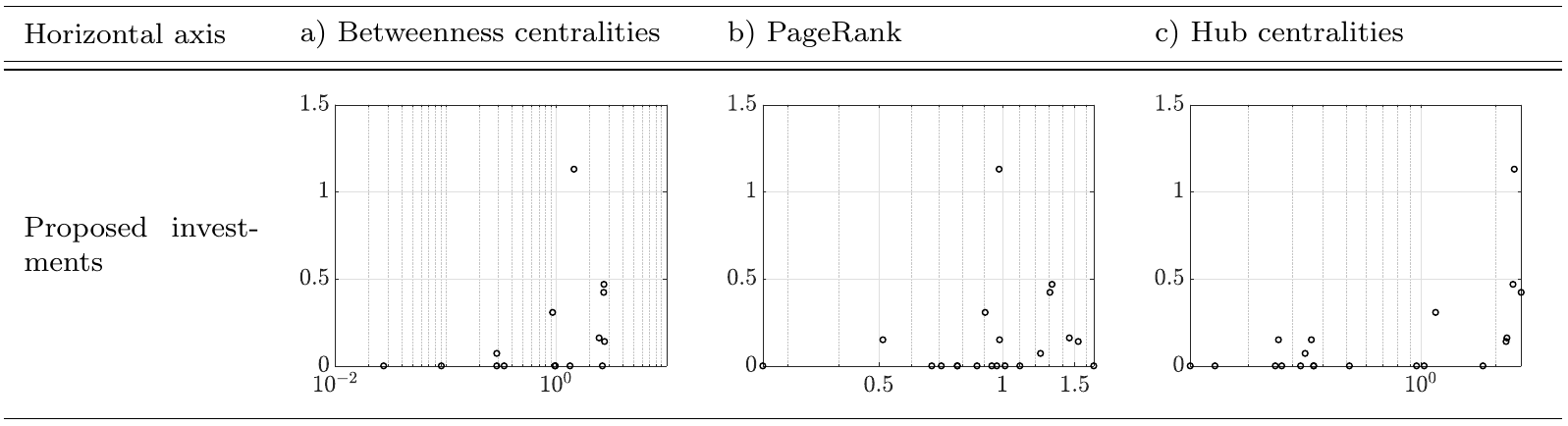}
\caption{Optimal investments versus centrality measures in the automotive appearance design}
\label{table:autoInvestments}
\end{figure*}

\begin{figure}[tb]
\centering
\includegraphics[height=\figheight]{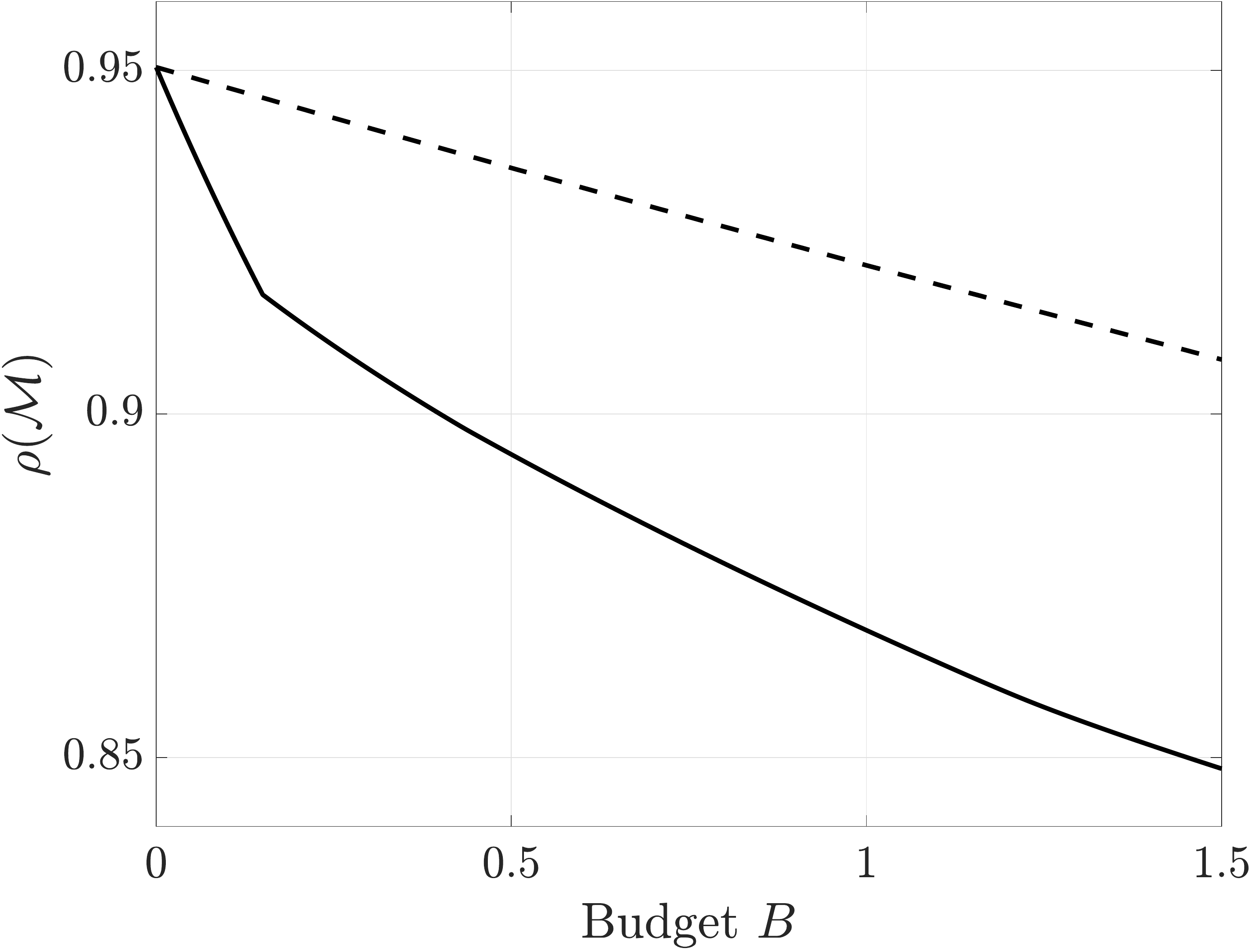}
\caption{Performances of the baseline and proposed strategies for various values of the budget~$B$. Solid line: proposed strategy. Dashed line: baseline strategy.}
\label{fig:changebudget}
\vspace{3mm}
\includegraphics[height=\figheight]{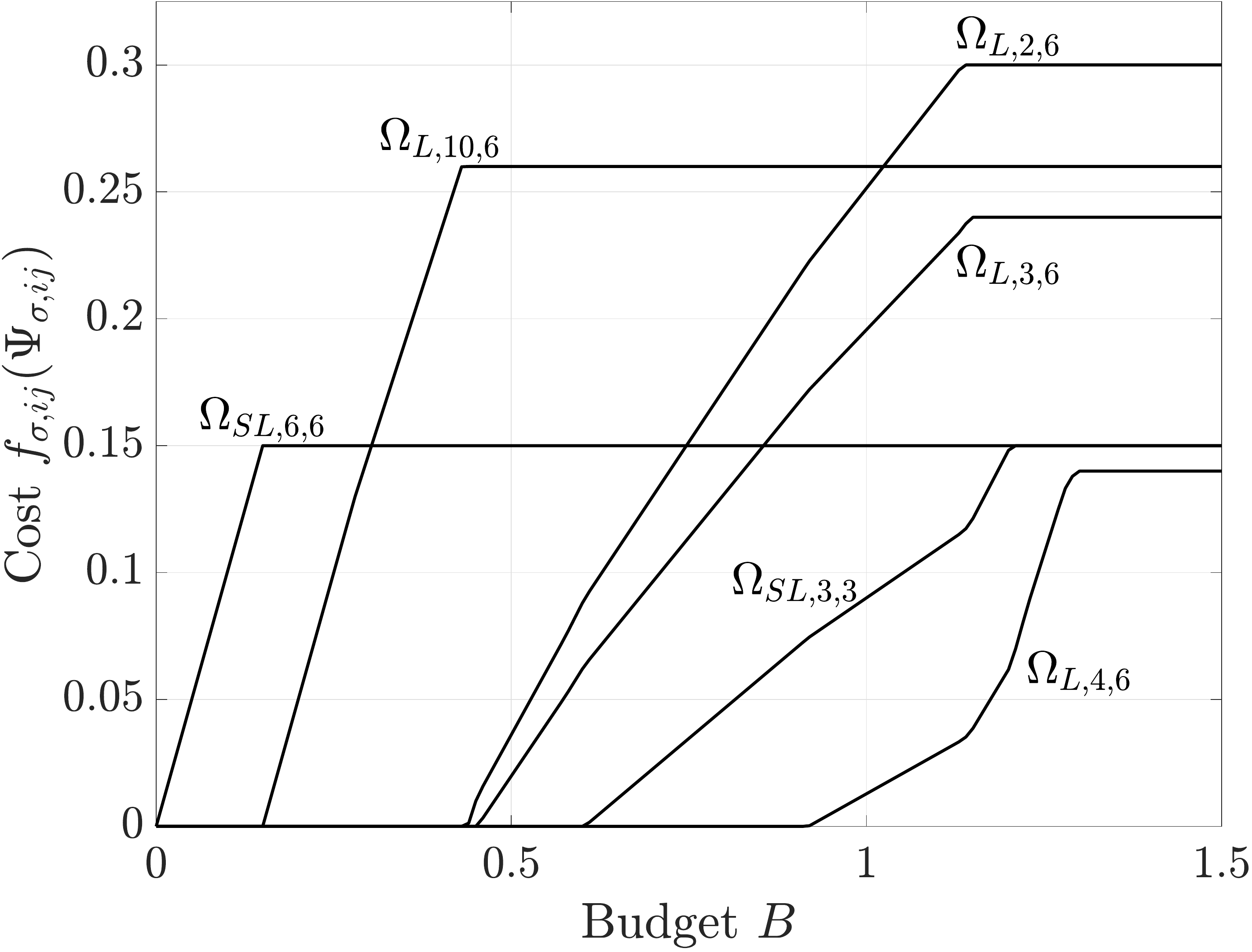}
\caption{{Investments on the dependencies versus the budget~$B$. We show the investments on the dependencies that receive the top six investments in Fig.~\ref{table:optInvestments}.}}
\label{fig:changebudgeteach}
\end{figure}

\begin{figure}[tb]
\centering
\includegraphics[height=\figheight]{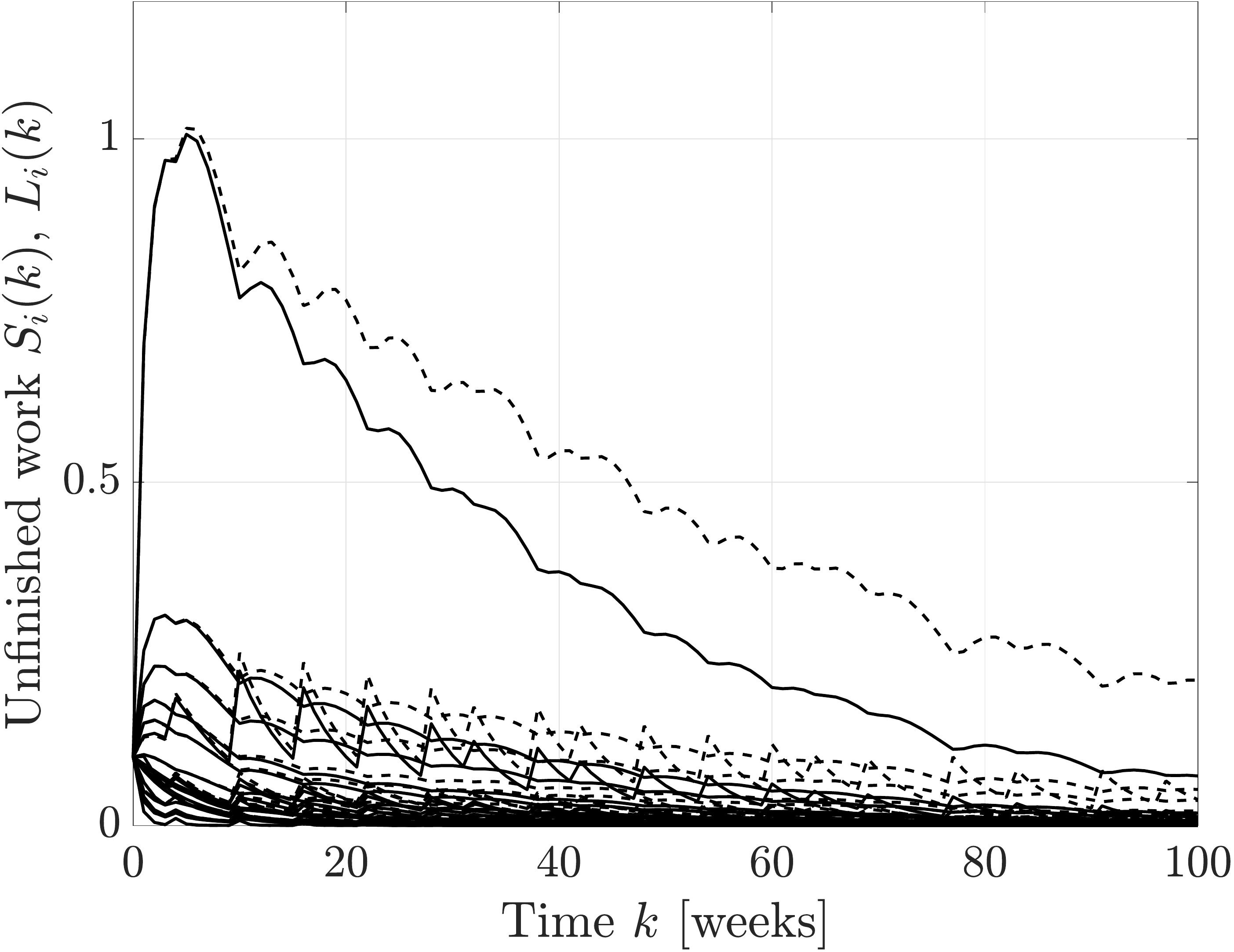}
\caption{Amount of unfinished work for each design component and teams. Solid lines: proposed strategy. Dashed lines: baseline strategy.}
\label{fig:samplePaths}
\vspace{3mm}
\centering
\includegraphics[height=\figheight]{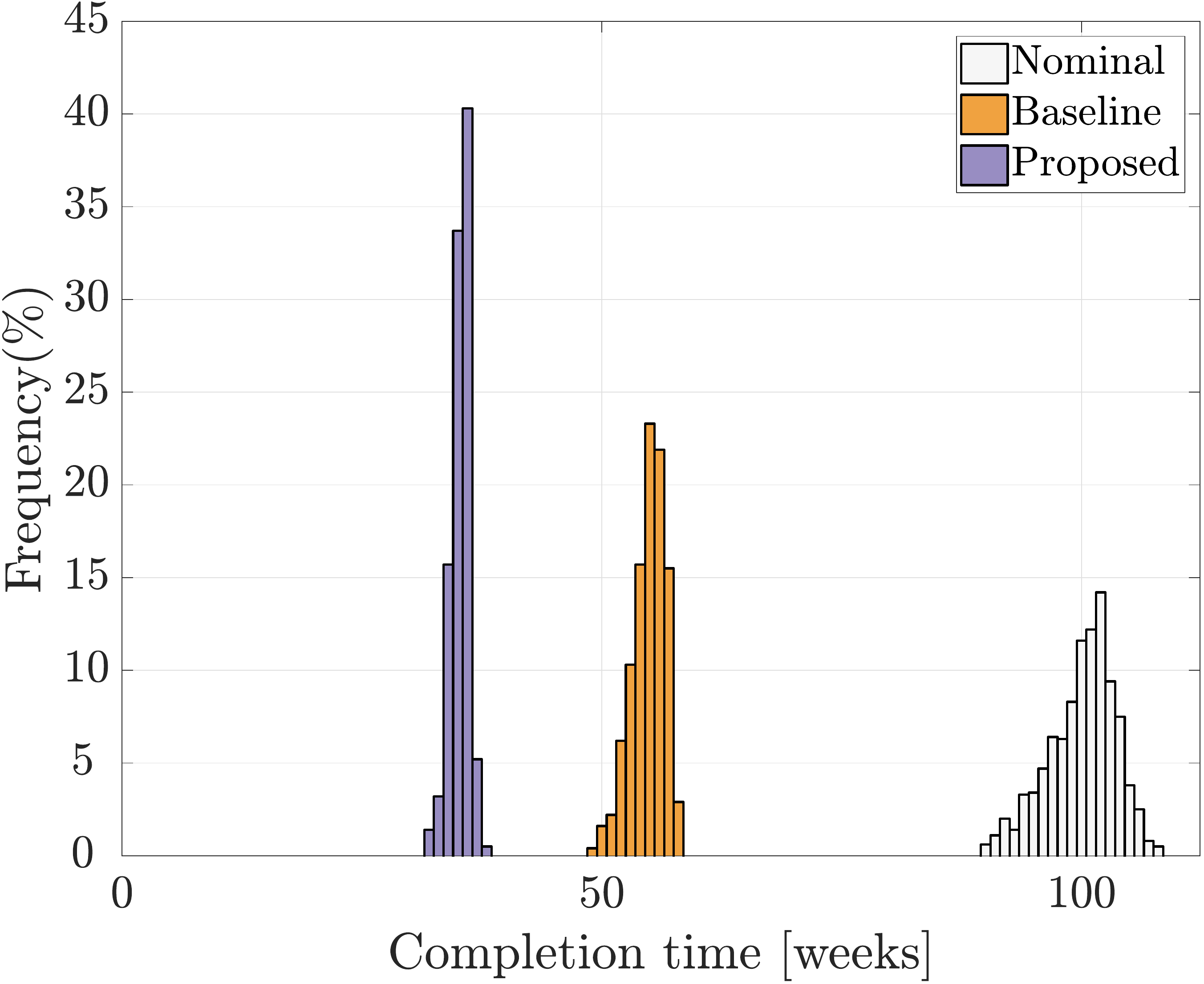}
\caption{{Histograms of the completion times of the nominal project, the project with the baseline investments, and the one with the proposed investments.}}
\label{fig:hist}
\end{figure}

\subsection{Results}\label{sec:results}

For simplicity of illustration, we first fix the probabilities~\eqref{eq:randomInerval} as $p_1= p_2 = p_3 = 0$, $p_4 = p_5 = 1/8$, $p_6 = 1/2$, and $p_7 = p_8 = 1/8$. We set $\epsilon = 0.85$ (i.e., a full investment eliminates 15\% of the nominal dependency) and solve the budget-constrained dependency optimization problem with the budget $B=1.5$. Using the baseline and proposed strategies, we obtain the two sets of resource distributions for weakening the strength of dependencies. The proposed strategy achieves the better performance~$\rho(\mathcal M) = 0.8484$ compared with the baseline strategy $\rho(\mathcal M) = 0.9079$. We show the obtained investments in Figs.~\ref{table:optInvestments} and~\ref{table:conInvestments}. A remarkable difference between the proposed and baseline strategies can be found in the investments on~$\Omega_{LS}$; although the proposed strategy invests no resource on~$\Omega_{LS}$, the baseline strategy spends more than the half of the budget in its improvement. 

\textcolor{\changesColor}{Let us then observe how the investment from the proposed strategy depends on the centrality metrics of the network underlying the PD process. For this purpose, we determine the aggregated amount of investments on an individual task by the total sum of investments on the dependencies involving the task. For example, the aggregated amount of investment on the task $L_9$ in the proposed strategy equals $(8.1+2.1+4.0+1.9)/100 = 0.161$. On the other hand, to compute the centrality metrics of each task, we construct the extended DSM by the formula
\begin{equation}\label{eq:def:Omega}
\Omega = \begin{bmatrix}
\Omega_{L} & \Omega_{LS}\\
\Omega_{SL} & \Omega_{S}
\end{bmatrix}
\end{equation}
and consider the directed network~$\mathcal G_{\Omega}$ whose adjacency matrix equals~$\Omega$. We then compute the following three different centrality measures for the network~$\mathcal G_{\Omega}$; the betweenness centrality, the PageRank, and the Hub centrality (see, e.g., \cite{Newman2010a}). Each centrality metric is normalized to sum to $2m = 20$. In Fig.~\ref{table:autoInvestments}, we show how the amount of aggregated investments in the proposed strategy depend on the aforementioned centrality measures. Interestingly, we find that the dependencies of the proposed investments on the centrality measures are not very clear for any cases, which contradicts our intuition that, the bigger centrality a task has, the more investment the task should receive. This observation indicates that, for the improvement of the automotive PD project, the centrality measures are not necessarily an indicative measure. We shall revisit and discuss this counterintuitive result in Section~\ref{sec:synth}.}

We next observe how the amount of the available budget~$B$ affects performance improvements. In Fig.~\ref{fig:changebudget}, we show the optimized values of the spectral radius $\rho(\mathcal M)$ from the proposed and baseline strategies as we increase $B$ from $0$ to $1.5$. We confirm that the proposed strategy achieves improvement on the baseline strategy for any value of the budget. Then, in Fig.~\ref{fig:changebudgeteach}, we show the investments on the top six dependencies in Fig.~\ref{table:optInvestments} for various values of the budget~$B$. We observe that the amount of investment on each dependency is not necessarily related to the priority of the investment. For example, although the dependency~$\Omega_{LS, 66}$ receives the second least amount of investments with the budget~$B=1.5$, the dependency receives all the available investments when the budget~$B$ is within the range~$[0, 0.15]$. This phenomenon indicates that the dependency~$\Omega_{LS, 66}$ should be invested with the top priority. This nontrivial pattern of the optimal investment indicates the necessity of an analytical tool for PD managers to develop an appropriate resource allocation strategy.

\begin{figure*}[tb]
\color{\changesColor}
\centering
\includegraphics[width=.93\linewidth]{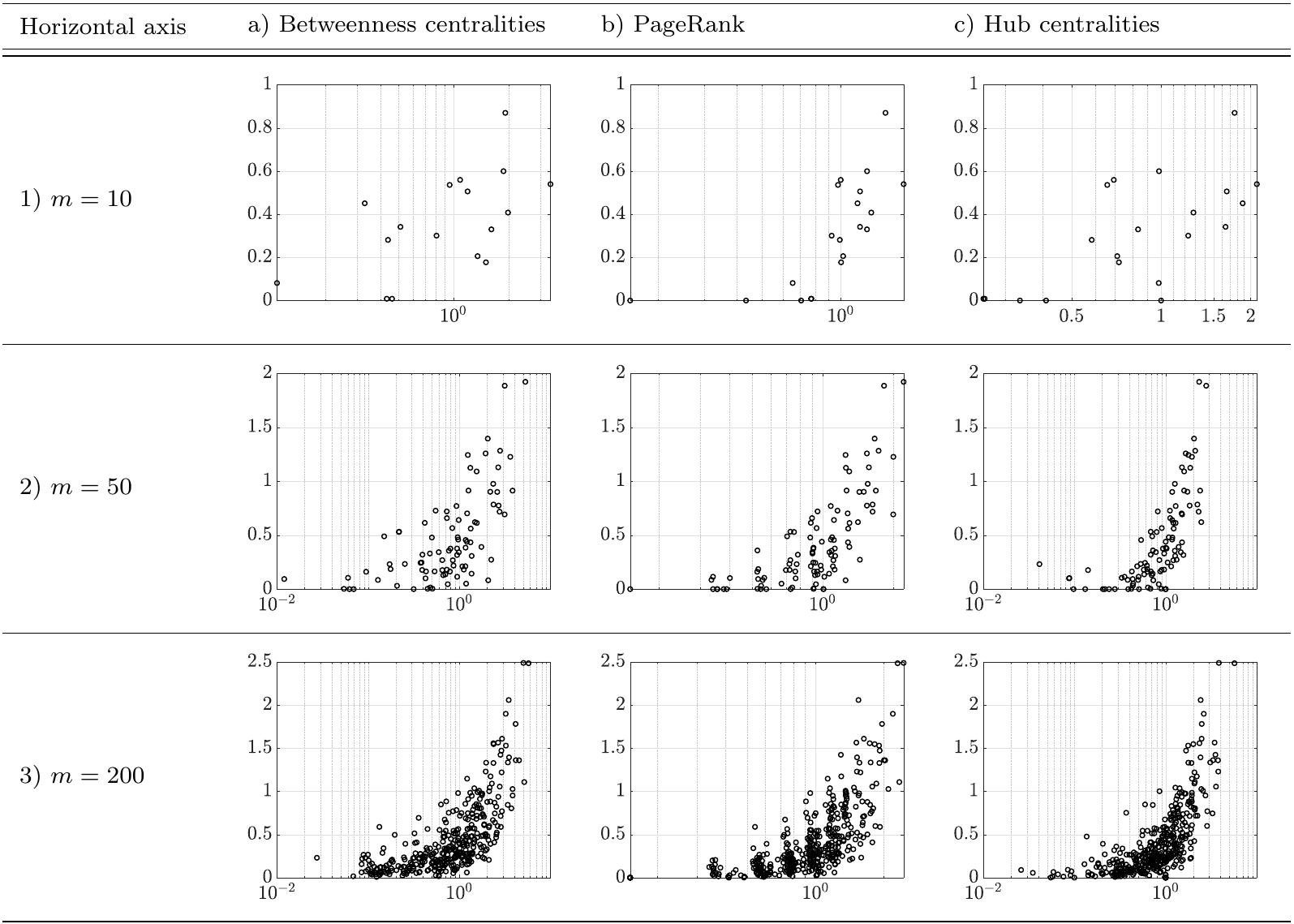}
\caption{Optimal investments versus centrality measures in the Erd\H{o}s-R\'enyi PD networks}
\label{table:ER}
\end{figure*}

\begin{figure*}[tb]
\color{\changesColor}
\centering
\includegraphics[width=.93\linewidth]{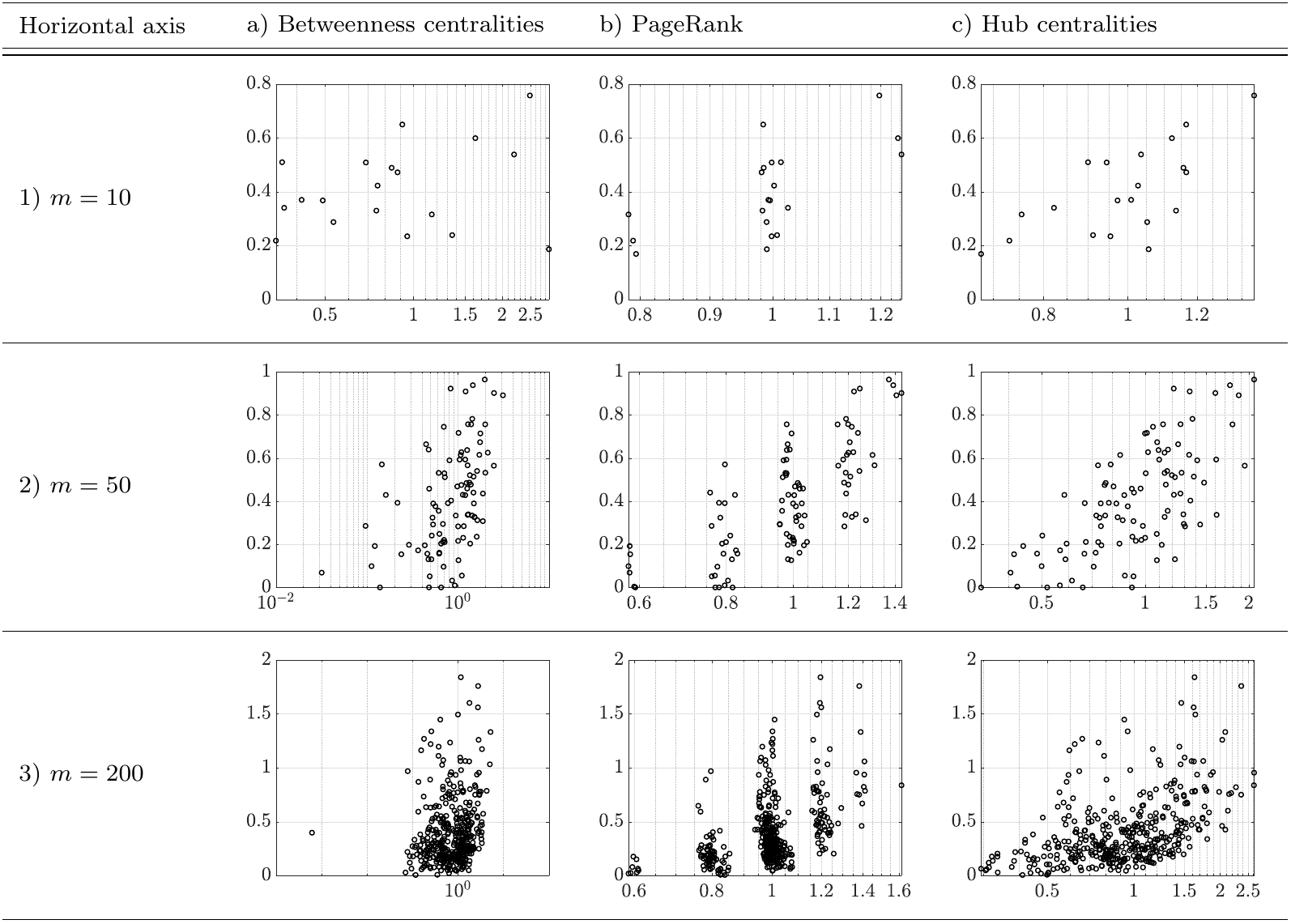}
\caption{Optimal investments versus centrality measures in the Watz-Strogatz PD networks}
\label{table:WS}
\end{figure*}

\begin{figure*}[tb]
\color{\changesColor}
\centering
\includegraphics[width=.93\linewidth]{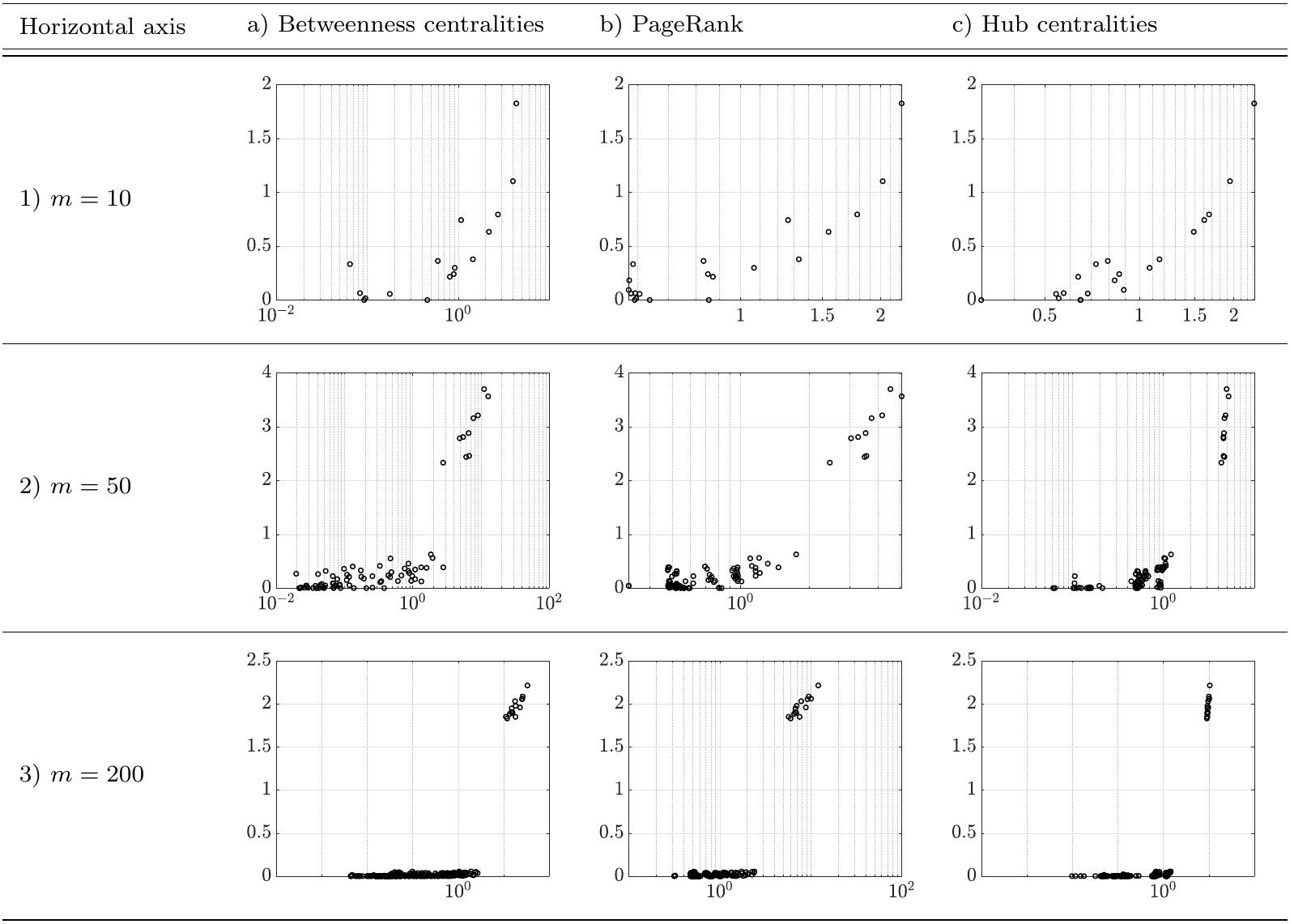}
\caption{Optimal investments versus centrality measures in the Barab\`asi-Albert PD networks}
\label{table:BA}
\end{figure*}

Let us finally compare the amounts of reductions in project completion times. In Fig.~\ref{fig:samplePaths}, we show the trajectories of the amount of unfinished work when the proposed and baseline strategies are applied. We see that the proposed strategy achieves a faster decay in the amount of unfinished work. In order to further investigate the sensitivity of this analysis, we perform the following experiment. Let us suppose that the design project completes if the total amount of unfinished works across the project falls below a specified level, say,~$\gamma$. In this case study, this level is set as $\gamma = 1$ (i.e., $1/20$ units of unfinished work within each task in average). In Fig.~\ref{fig:hist}, we show the histogram of completion times of the nominal project, the project with the baseline investments, and the one with the proposed investments. To create the histogram, we randomly draw the numbers~$(p_4, p_5, p_6, p_7, p_8)$ from a uniform distribution on the set~$\{(p_4, p_5, p_6, p_7, p_8)\colon p_4\geq 0,\, p_5\geq 0,\,p_6\geq 0,\,p_7\geq 0,\,p_8\geq 0,\,p_4+p_5+p_6+p_7 + p_8 = 1\}$. By following this procedure, we randomly generate a thousand probability distributions for the feedback intervals and, for each distribution, we find the resource allocation using the proposed and baseline strategies. We then simulate the design process and obtain one thousand sample paths for the nominal project, the project with the baseline investments, and the one with the proposed investments, from which we compute the completion times. From Fig.~\ref{fig:hist}, we see that the proposed strategy can effectively reduce the project completion times.

\color{\changesColor}

\section{Centrality metrics}\label{sec:synth}

In the case study that we performed in the previous section, we have observed that a task having a larger centrality does not necessarily receive a bigger investment at the optimality (see Fig.~\ref{table:autoInvestments}). On the other hand, it was shown in the reference~\cite{Braha2007} that the dynamics and performance of engineering networks are dominated by a few highly central nodes. Specifically, it was observed in the reference that we can exploit such connectivity patterns of complex design networks to suppress the amplification and propagation of design changes and errors through the network. In order to explain this apparent contradiction and provide reliable and practical guidelines for resource investments in engineering design networks, in this section we analyze how the nominal boundary conditions on the DSMs and IDMs, i.e., the connectivity of the design networks, affects the dependence of the optimal investments on centrality measures.

For this purpose, we consider the cases where the structure of the PD network is determined by the following three major network models; the Erd\H{o}s-R\'enyi \cite{ErdHos1959}, Watz-Strogatz \cite{Watts1998}, and  Barab\'asi-Albert \cite{Barabasi1999} graphs. The Erd\"os-R\'enyi graph is one of the earliest random graph models and commonly used in the literature, partly due to the fact that the model is easy to analyze. On the other hand, the Watz-Strogatz and Barab\'asi-Albert graphs are more practical network models able to reproduce, respectively, the small-world and scale-free properties,  which are frequently observed in empirical complex networks. For more detailed accounts on the graph models from the perspective of engineering design, we refer the readers to the references~\cite{Braha2004a,Braha2004}.

For each type of graph models, we construct a sample network (i.e., a realization) having $2m$ nodes for $m=10$, $50$, and $200$, respectively (recall that the PD process under our consideration contains $m$ components and each component is separately assigned into the local and system teams). Then, we set the extended DSM of the PD process, given in \eqref{eq:def:Omega}, by the formula~$\Omega = cA$, where $A$ is the adjacency matrix of the sample network and the constant~$c > 0$ is chosen in such a way that the generalized WTM, given by~\eqref{eq:complexMat}, has the spectral radius one, i.e., $\rho(\mathcal M) =1$. Under this setting, the nominal PD project without resource allocations is (marginally) not feasible. For this nominal PD project, we solve the convex optimization problem~\eqref{eq:origOptDMSIDM:pc} to obtain the optimal investments that solve the budget-constrained dependency optimization problem.  In this optimization, we use the parameter~$\epsilon = 0.5$ and set the budget~$B$ to be 10\% of the full resource allocation. Also, we normalize the sum of the centralities to be equal to~$2m$, as was done in the case study in the previous section.

In Fig.~\ref{table:ER}, we show the amount of the investments for individual tasks versus the centralities of the network~$\mathcal G_{\Omega}$ for the case of Erd\H{o}s-R\'enyi PD networks. We can observe a tendency that, when the network size is relatively large ($m=50$ and $m=200$), tasks having bigger centralities tend to receive heavier investments. However, this tendency is not very strong for us to decide on investing towards only on the tasks having higher centralities, as there still exist some tasks receiving relatively heavy investments but having smaller centralities. We also note that, when the network size is small ($m=10$), the amounts of investments seem not to have a significant correlation with any of the centrality measures. The above observations suggest that, when the network underlying a PD process exhibits a random-graph like structure, the proposed investment strategy could outperform the ones purely based on centrality measures. We can also draw a similar conclusion from Fig.~\ref{table:WS} for the Watz-Strogatz network model.

However, the guidelines for the cases of Erd\H{o}s-R\'enyi and Watz-Strogatz PD networks do not immediately extend to the last network model of Barab\'asi-Albert graphs. In Fig.~\ref{table:BA}, we show the amount of investments on individual development tasks versus the centralities in the case of the Barab\`asi-Albert graph. In contrast with the former two cases (Figs.~\ref{table:ER} and~\ref{table:WS}), we see that the amounts of investments almost monotonically increase with the centralities, even for the small network size of $m=10$. Furthermore, in the case of $m=200$, most of the tasks receive almost zero investments, while a few high-centrality nodes are heavily invested for the improvement of the PD process, as previously reported in the reference~\cite{Braha2007}.

Let us summarize our findings as follows. As can be observed from the second and third rows in Fig.~\ref{table:BA}, if a PD network has a scale-free property and is relatively large-scale, then heavily investing on high-centrality tasks would yield the best outcome, which coincides with findings in the literature. However, this natural investment strategy may not work well when the underlying PD networks are rather random or have a small-world property, as can be seen from Figs.~\ref{table:ER} and~\ref{table:WS}. In such a case, the project managers are encouraged to consider not only high-centrality tasks but also the ones having `middle' centralities to obtain the best outcome from investing in their PD projects. In fact, the PD network of our case study in Section~\ref{sec:caseStudy} falls into this category; as Fig.~\ref{table:DSMIDM} indicates, the underlying network structure is neither large-scale nor scale-free. This characteristics particularly explains the small correlation between the amount of the optimal investments and the centralities of the individual tasks.
\color{black}

\section{Conclusion and discussion} \label{sec:con}

In this paper, we have presented an analytical framework for optimizing the feasibility of PD projects having a time-varying architecture. We have focused on the specific but widely-observed situation in which system and local team asynchronously interact for development. By using a result from the systems and control theory, we first have derived the feasibility index of the project for determining the long-term feasibility of the project. Building on this analysis result, we have shown that the optimal resource distribution solving the budget-constrained and performance-constrained dependency optimization problems can be efficiently found by solving convex optimization problems. The effectiveness of the framework has been illustrated by the case study taking an automobile PD project as an example. \textcolor{\changesColor}{We have furthermore analyzed how the boundary conditions on the DSMs and IDMs affect the optimal investment patterns by using major graph models from the Network Science.}

It has long been implicitly assumed in most of the resource-allocation support tools that the dependency structure between modules is static (i.e., does not change over time) in PD projects. However, in several practical situations, the dependency structure experiences dynamical changes over time due to various endogenous and exogenous reasons. Despite this fact, little research effort has been spent on the issue of making optimal decisions on resource allocation for developing design rules. This is in great contrast with the research trend in the network science, in which the time-variability of network structures has been attracting growing attention~\cite{Holme2015b,Masuda2016b} \textcolor{\changesColor}{since their discovery~\cite{Braha2006,Braha2009}}. This paper is the first to shed a light on the issue somewhat overlooked in design engineering.

Several limitations of the proposed framework suggest fruitful directions for future research. One of the limitations of the current study stems from the fact that the framework allows the case in which the project has one and only one system and local teams. This limitation is mostly due to that of regenerative processes; if there are multiple system or local teams, then one can confirm that the work transformation between the teams cannot be described by a regenerative process. However, by employing more general results from systems and control theory such as the one in~\cite{Ogura2015}, it would become possible to deal with such a general case.

Data availability is another issue that should be strengthened in future research. There exist established methodologies for the measurement of DSMs and WTMs in the literature (see, e.g., \cite{Browning2016}). On the other hand, we have not addressed the details on measuring the distribution \eqref{eq:randomInerval} of the interval between feedbacks. If past historical data of the same or similar projects is available to the PD manager, then it is realistic to estimate the distribution to a certain precision. However, if the PD project is completely new and, therefore, such past history is not available, then we would have to rely on a rough estimate by the PD manager, which could deteriorate the reliability of our framework. One of the possible approaches to solve this issue is to use robust control theoretical tools on switched linear systems (see. e.g., \cite{Lin2007,Ogura2015a}).

The current framework is tailored to large-scale, complex, and dynamical PD project in which communications between teams does not necessarily regularly occur but are asynchronous. For this reason, the current framework is not directly applicable to other widely-used frameworks such as Scrum for software development processes, in which communications between development members occur continuously and, therefore, the dependency structure within the project is not likely to change dynamically. However, it has been demonstrated in the network science literature that ignoring even a small temporality in dependency structure can lead to an ineligible error in our analysis~\textcolor{\changesColor}{\cite{Fefferman2007,Vazquez2007,Braha2006,Hill2010b}}. Such findings suggest that the currently widespread paradigm of ``static dependency'' \textcolor{\changesColor}{for PD projects improvements} may have to be carefully reconsidered.

\end{document}